\documentclass[11pt,oneside,english,dvipsnames]{amsart}
\usepackage[T1]{fontenc}
\usepackage[latin9]{inputenc}
\usepackage{mathtools}
\usepackage{amstext}
\usepackage{amsthm}
\usepackage{amssymb}
\usepackage[a4paper]{geometry}
\usepackage[all]{xy}

\makeatletter

\providecommand{\tabularnewline}{\\}

\numberwithin{equation}{section}
\numberwithin{figure}{section}
\theoremstyle{plain}
\newtheorem{thm}{\protect\theoremname}
\theoremstyle{definition}
\newtheorem{defn}[thm]{\protect\definitionname}
\theoremstyle{plain}
\newtheorem{lem}[thm]{\protect\lemmaname}
\theoremstyle{remark}
\newtheorem{rem}[thm]{\protect\remarkname}
\theoremstyle{definition}
\newtheorem{example}[thm]{\protect\examplename}
\theoremstyle{plain}
\newtheorem{cor}[thm]{\protect\corollaryname}
\theoremstyle{remark}
\newtheorem{notation}[thm]{\protect\notationname}
\theoremstyle{plain}
\newtheorem*{thm*}{\protect\theoremname}

\date{}
\usepackage{todonotes}

\usepackage{tikz}
\usetikzlibrary{decorations.markings}
\usepackage{multicol}

\theoremstyle{definition}
\newtheorem{construction}[thm]{Construction}

\usepackage[bookmarks=true,
            bookmarksopen=true,
            pagebackref=true,
            hyperindex=true,
            colorlinks=true,
            linkcolor=blue,
            citecolor=blue,
            filecolor=blue,
            urlcolor=blue,
			pdftitle={On the recognition problem for limits of entropy functions}
            ]{hyperref}

\usepackage[nameinlink]{cleveref}

\makeatletter
\DeclareFontFamily{OMX}{MnSymbolE}{}
\DeclareSymbolFont{MnLargeSymbols}{OMX}{MnSymbolE}{m}{n}
\SetSymbolFont{MnLargeSymbols}{bold}{OMX}{MnSymbolE}{b}{n}
\DeclareFontShape{OMX}{MnSymbolE}{m}{n}{
    <-6>  MnSymbolE5
   <6-7>  MnSymbolE6
   <7-8>  MnSymbolE7
   <8-9>  MnSymbolE8
   <9-10> MnSymbolE9
  <10-12> MnSymbolE10
  <12->   MnSymbolE12
}{}
\DeclareFontShape{OMX}{MnSymbolE}{b}{n}{
    <-6>  MnSymbolE-Bold5
   <6-7>  MnSymbolE-Bold6
   <7-8>  MnSymbolE-Bold7
   <8-9>  MnSymbolE-Bold8
   <9-10> MnSymbolE-Bold9
  <10-12> MnSymbolE-Bold10
  <12->   MnSymbolE-Bold12
}{}

\let\llangle\@undefined
\let\rrangle\@undefined
\DeclareMathDelimiter{\llangle}{\mathopen}%
                     {MnLargeSymbols}{'164}{MnLargeSymbols}{'164}
\DeclareMathDelimiter{\rrangle}{\mathclose}%
                     {MnLargeSymbols}{'171}{MnLargeSymbols}{'171}
\makeatother

\crefname{construction}{construction}{constructions}
\Crefname{construction}{Construction}{Constructions}
\Crefformat{construction}{#2Construction #1#3}
\crefformat{construction}{#2construction #1#3}

\makeatother

\usepackage{babel}
\providecommand{\corollaryname}{Corollary}
\providecommand{\definitionname}{Definition}
\providecommand{\examplename}{Example}
\providecommand{\lemmaname}{Lemma}
\providecommand{\notationname}{Notation}
\providecommand{\remarkname}{Remark}
\providecommand{\theoremname}{Theorem}

\begin{document}
\title{On the recognition problem for limits of entropy functions}
\author{Geva Yashfe}
\address{University of Chicago}
\begin{abstract}
We prove that there is no algorithm to decide whether a given integer
vector is in the closure of the entropic cone $\overline{\Gamma_{n}^{*}}$.
Equivalently, there is no decision procedure to determine whether
a given integer-valued function $h:\mathcal{P}\left(\{1,\ldots,n\}\right)\rightarrow\mathbb{Z}_{\ge0}$
is a pointwise limit of joint entropy functions. In other words, given
such an $h$, it is undecidable whether for all $\varepsilon>0$ there
exists a finite probability space $\left(\Omega,P\right)$ with random
variables $X_{1},\ldots,X_{n}$ such that their joint entropy $H$
satisfies $\max_{I\subseteq\{1,\ldots,n\}}\left|H\left(X_{I}\right)-h\left(I\right)\right|<\varepsilon${\normalsize .}
This settles the last open case in a sequence of related undecidability
results proved by L. K\"{u}hne and the author, with applications
in algorithmic information theory. The main new tool is a Desargues'-type
theorem for almost entropic polymatroids.
\end{abstract}

\maketitle

\section{Introduction}

The entropy regions $\Gamma_{n}^{*}\subset\mathbb{R}^{2^{n}}$ are
central objects of algorithmic information theory. For given $n\in\mathbb{N}$,
they describe all the possible joint entropies of a collection of
$n$ random variables on a finite probability space (for precise definitions
see \cref{sec:entropy}; for context in algorithmic information theory,
see \cite{Yeung_network_coding}). The closures $\overline{\Gamma_{n}^{*}}$
of these regions are mathematically better behaved, being convex cones,
and in a sense they are more useful in practice. The fact that $\Gamma_{n}^{*}\neq\overline{\Gamma_{n}^{*}}$
for all $n\ge4$ captures the phenomenon that some network coding
problems have no optimal solution, but rather a sequence of successively
better solutions approaching some optimal rate.

Our main result is the following.
\begin{thm}
\label[theorem]{thm:undecidability}There is no decision procedure
that takes as input $n\in\mathbb{N}$ and a vector $v\in\mathbb{Z}^{2^{n}}$
and determines whether $v\in\overline{\Gamma_{n}^{*}}$.
\end{thm}

We prove this by showing the equivalent statement that there is no
algorithm that determines whether a given (finite) matroid is almost
entropic.

Several implications of this are known. A primary one is that approximate
conditional independence implication is undecidable. For a detailed
discussion of approximate conditional independence implication, see
\cite{Kenig_Suciu}; the proof of undecidability given \cref{thm:undecidability}
is given in \cite[Thm. 8.3]{KY22}.

\subsection*{Related work}

The non-approximate conditional independence implication problem was
first proved undecidable (together with several related problems)
in \cite{CT_Li}. This was also independently shown in \cite{KY22},
that first proved a form of \cref{thm:undecidability} for $\Gamma_{n}^{*}$,
as well as for ``almost multilinear'' polymatroids (that correspond
to approximate coding with linear codes). The present paper strongly
depends on the results of \cite{KY22} on almost multilinear matroids
(see \cref{sec:KY22}).

The cones $\overline{\Gamma_{n}^{*}}$ have been considered in various
papers from a dual viewpoint: since they are closed convex cones in
Euclidean space, they are characterized by the linear inequalities
that they satisfy. There have been various attempts to understand
these inequalities. Some of the first results in this direction are
\cite{Matus_non_polyhedral}, \cite{Zhang_Yeung}, \cite{DFZ1}.

\subsection*{Organization of the paper}

We recall and introduce basic definitions in \cref{sec:preliminaries}.
In \cref{sec:aent_geometry} we explain how to do synthetic geometry
in an almost entropic setting, and prove the version of Desargues'
theorem that we need. In \cref{sec:PDG} we recall the partial Dowling
geometries (PDGs) of \cite{KY19} and \cite{KY22} (in \cite{KY19}
these were called \emph{generalized} Dowling geometries): these are
the central objects we work with. In \cref{sec:PDG_groups} we explain
how to recover a group from an almost entropic PDG of rank $4$, and
in \cref{sec:PDG_lifting} we explain how the same can be done with
an almost entropic PDG of rank $3$. The main theorem is proved in
\cref{sec:undecidability}.

\subsection*{Acknowledgements}

It is a pleasure to thank many people who helped, both directly and
indirectly, with the writing of this paper. I thank Lukas K\"{u}hne
for many helpful conversations, and for our collaboration in previous
works which led to these results. I also thank the organizers and
participants of the Dagstuhl Seminar 22301 ``Algorithmic Aspects
of Information Theory,'' who taught me some of the basic notions
of the field, and without whom this paper could not have been written.

Finally, I thank Zlil Sela for helpful conversations and support, as well as the math department of the Hebrew University of Jerusalem. This paper was written while I was a PhD student in Jerusalem. 

\section{Preliminaries}

\label[section]{sec:preliminaries}

\subsection{Matroids and polymatroids}
\begin{defn}
A\emph{ polymatroid} is a pair $\left(E,f\right)$ consisting of a
set $E$ (called the \emph{ground set}) together with a function $f:P\left(E\right)\rightarrow\mathbb{R}$
satisfying:
\begin{enumerate}
\item $f\left(\emptyset\right)=0$,
\item If $A\subseteq B\subseteq E$ then $f\left(A\right)\le f\left(B\right)$
(monotonicity),
\item If $A,B\subseteq E$ then $f\left(A\right)+f\left(B\right)\ge f\left(A\cup B\right)+f\left(A\cap B\right)$
(submodularity).
\end{enumerate}
The three conditions above are the polymatroid axioms. 

If $E$ is infinite, then $f$ is of \emph{finite type} if it also
satisfies:
\begin{enumerate}
\item [4.] For all $S\subseteq E$ we have $f\left(S\right)=\sup\left\{ h\left(F\right)\mid F\subseteq S\text{ finite}\right\} $.
\end{enumerate}
It is a \emph{matroid} if in addition $f$ takes values in $\mathbb{Z}$
and $f(\{x\})\le1$ for all $x\in E$.

For $a_{1},\ldots,a_{k}\in E$ we denote $f\left(a_{1},\ldots,a_{k}\right)=f\left(\left\{ a_{1},\ldots,a_{k}\right\} \right)$.
\end{defn}

We collect basic consequences for readers less familiar with this
language.
\begin{lem}
[Basic consequences of the axioms] Let $\left(E,f\right)$ be a polymatroid.
Then:
\begin{enumerate}
\item (Diminishing returns.) For any $B\subseteq A\subseteq E$ and $C\subseteq E$
such that $A\cap C=\emptyset$ we have $f\left(B\cup C\right)-f\left(B\right)\ge f\left(A\cup C\right)-f\left(A\right)$.
This property is equivalent to submodularity.
\item (Closure.) For $A\subseteq E$ define $\mathrm{cl}\left(A\right)=\left\{ x\in E\mid f\left(A\right)=f\left(A\cup\left\{ x\right\} \right)\right\} $
(this is the closure operator). If $\left(E,f\right)$ is finite type
then $f\left(A\right)=f\left(\mathrm{cl}\left(A\right)\right)$ and
$\mathrm{cl}\left(\mathrm{cl}\left(A\right)\right)=\mathrm{cl}\left(A\right)$.
\item (Subadditivity.) For any $A_{1},\ldots,A_{n}\subseteq E$ we have
$f\left(\bigcup_{i=1}^{n}A_{i}\right)\le\sum_{i=1}^{n}f\left(A_{i}\right)$.
\end{enumerate}
\end{lem}

\begin{proof}
\begin{enumerate}
\item Apply submodularity to the sets $B\cup C$ and $A$ and apply the
assumptions on $A,B,C$ to obtain $f\left(B\cup C\right)+f\left(A\right)\ge f\left(A\cup C\right)+f\left(B\right)$,
then rearrange the terms. 

To see that conversely the diminishing returns property also implies
submodularity, let $X,Y\subseteq E$. Define $C\coloneqq X\setminus Y$,
$B=X\cap Y$, and $A=Y$. Then apply diminishing returns (noting that
$B\subseteq A$ and $A\cap C=\emptyset$ by construction).
\item Observe that $A\subseteq\mathrm{cl}\left(A\right)$. Apply monotonicity
and diminishing returns to obtain that $f\left(A\right)=f\left(A\cup B\right)$
for any finite $B\subseteq\mathrm{cl}\left(A\right)$, by induction
on $\left|B\right|$. Then apply the finite type assumption. This
shows $f\left(A\right)=f\left(\mathrm{cl}\left(A\right)\right)$.
Given this and $A\subseteq\mathrm{cl}\left(A\right)$, the equality
$\mathrm{cl}\left(\mathrm{cl}\left(A\right)\right)=\mathrm{cl}\left(A\right)$
follows directly.
\item By induction on $n$, using submodularity.
\end{enumerate}
\end{proof}
\begin{rem}
By a polymatroid $f$ \emph{on $E$ }we mean a polymatroid $\left(E,f\right)$.
We sometimes omit $E$ and refer just to $f$ if the ground set is
clear. Such an $f$ can be identified with a vector in $\mathbb{R}^{2^{E}}$.
The collection of all polymatroids on $E$ is then identified with
a convex cone; if $E$ is finite, this cone is closed.
\end{rem}

\begin{rem}
The ground set $E$ of a (poly-)matroid is sometimes required to be
finite, but it is convenient to allow infinite $E$ in this paper.
There are various notions of infinite matroids, suitable for different
applications. I was unable to find the ``finite type'' hypothesis
used here elsewhere in the literature, but it is close to standard
notions: specialized to matroids, it is the same as requiring that
the matroid is finitary and of finite rank.

An example of a non-finitary matroid $\left(\mathbb{Q},f\right)$
is given by 
\[
f\left(S\right)=\begin{cases}
1 & \sqrt{2}\in\mathrm{cl}_{\mathbb{R}}\left(S\right)\\
0 & \text{otherwise,}
\end{cases}
\]
where $\mathrm{cl}_{\mathbb{R}}$ is the usual topological closure
operation.

One can tighten condition (4) slightly and demand that for all $S\subseteq E$
there exists a finite $F\subseteq S$ such that $f\left(F\right)=f\left(S\right)$.
But it's not clear whether our constructions can be made to satisfy
this.
\end{rem}

\begin{defn}
An \emph{extension} of a polymatroid $\left(E,f\right)$ is a polymatroid
$\left(\hat{E},f\right)$ such that $E\subseteq\hat{E}$ and $\hat{f}\restriction_{\mathcal{P}\left(E\right)}=f$.
A \emph{restriction} of $f$ is a polymatroid $f'$ such that $f$
is an extension of $f'$.
\end{defn}

The preceding definition is very standard. We generalize it slightly
(in a rather obvious way) for later use.
\begin{defn}
An \emph{embedding} of a polymatroid $\left(E,f\right)$ into a polymatroid
$\left(\hat{E},f\right)$ is an injective function $\eta:E\rightarrow\hat{E}$
such that $f\left(S\right)=g\left(\eta\left(S\right)\right)$ for
all $S\subseteq E$.
\end{defn}

\begin{defn}
Let $\left(E,f\right)$ be a polymatroid and let $A\subseteq E$.
The function $f_{/A}:\mathcal{P}\left(E\right)\rightarrow\mathbb{R}$
given by
\[
f_{/A}\left(S\right)=f\left(S\cup A\right)-f\left(A\right)
\]
is a polymatroid, called the\emph{ contraction} of $\left(E,f\right)$
by $A$. We sometimes call $f_{/A}\left(S\right)$ ``the rank of
$S$ over $A$''.
\end{defn}

We may sometimes use other common definitions from matroid theory.
Undefined terminology can be found in \cite{Oxley_Matroid_theory}.

\subsection{Entropy, polymatroids, and the entropic cone}

\label[section]{sec:entropy}
\begin{defn}
Let $\left(\Omega,P\right)$ be a finite probability space and $X$
a random variable on $\Omega$ (with codomain $\mathrm{codom}\left(X\right)$).
The \emph{entropy}\footnote{Intuitively, entropy measures information content. Consider the distribution
of a tuple of $n$ independent samples from the distribution of $X$:
in order to encode the value of such a tuple (in such a way that it
can be decoded with high probability) roughly $nH\left(X\right)$
bits are both necessary and sufficient. This was made precise and
proved in \cite[Thm. 4]{Shannon_1948}.} of $X$ is
\[
H\left(X\right)=\sum_{x\in\mathrm{codom}\left(X\right)}P\left(X=x\right)\log\left(\frac{1}{P\left(X=x\right)}\right),
\]
where if $P\left(X=x\right)=0$ then $P\left(X=x\right)\log\left(1/P\left(X=x\right)\right)=0$
by convention. For concreteness we take the basis of the logarithm
to be $2$, but any real number larger than $1$ will work as well.
Note that $H\left(X\right)$ depends only on the distribution of $X$
(meaning the multiset of values $\left\{ P\left(X=x\right)\mid x\in\mathrm{codom}\left(X\right)\right\} $.
Also note that the sum always converges, since only finitely many
values have positive probability.

Let $\left\{ X_{i}\right\} _{i\in E}$ be random variables on $\Omega$
(with codomains $\left\{ \mathrm{codom}\left(X_{i}\right)\mid i\in E\right\} $).
For $S\subseteq E$ we denote by $X_{S}=\left(X_{i}\right)_{i\in S}$
the random variable $\Omega\rightarrow\prod_{i\in S}\mathrm{codom}\left(X_{i}\right)$
given by $X_{S}\left(\omega\right)=\left(X_{i}\left(\omega\right)\right)_{i\in S}$.
We obtain a function
\[
h:\mathcal{P}\left(E\right)\rightarrow\mathbb{R}
\]
\[
h\left(S\right)=H\left(X_{S}\right).
\]
It is known that this is a polymatroid; such polymatroids are called
\emph{entropic}.
\end{defn}

\begin{rem}
Contractions are often used in information theory, though with slightly
different language. For $\left(E,h\right)$ as above the value $h_{/A}\left(S\right)$
is often denoted $H\left(X_{S}\mid X_{A}\right)$ and is called the\emph{
conditional entropy}.
\end{rem}

\begin{defn}
Let $E$ be a set. A polymatroid $\left(E,h\right)$ of finite type
is \emph{almost entropic} if it is a pointwise limit of entropic polymatroids
on $E$, i.e. there are entropic polymatroids $\left\{ \left(E,h_{n}\right)\right\} _{n\in\mathbb{N}}$
such that $h_{n}\underset{n\rightarrow\infty}{\rightarrow}h$ pointwise.
\end{defn}

\begin{lem}
\label[lemma]{lem:aent_chain}Let $E_{1}\subseteq E_{2}\subseteq\ldots\subseteq E_{n}\subseteq\ldots$
be a chain of finite sets, $E=\bigcup_{n\in\mathbb{N}}E_{n}$, and
$\left(E,h\right)$ a polymatroid of finite type. Then $h$ is almost
entropic if and only if $h\restriction_{\mathcal{P}\left(E_{n}\right)}$
is almost entropic for all $n\in\mathbb{N}$.
\end{lem}

\begin{proof}
In one direction, it is immediate from the definitions that $h\restriction_{\mathcal{P}\left(E_{n}\right)}$
is almost entropic if $h$ is. 

In the other direction, one uses a diagonal argument: for each $n\in\mathbb{N}$
choose an entropic polymatroid $h_{n}$ such that $\sup_{S\subseteq E_{n}}\left|h_{n}\left(S\right)-h\left(S\right)\right|<\frac{1}{n}$.
Each $h_{n}$ comes with a finite probability space $\left(\Omega_{n},P_{n}\right)$
and random variables $\left\{ X_{n,i}\right\} _{i\in E_{n}}$. We
extend the collection of random variables by defining, for each $i\in E\setminus E_{n}$,
a constant random variable $X_{n,i}$ on $\Omega_{n}$. This extends
$h_{n}$ to a polymatroid $\hat{h}_{n}:\mathcal{P}\left(E\right)\rightarrow\mathbb{R}$
satisfying that $\hat{h}_{n}\left(S\right)=h_{n}\left(S\cap E_{n}\right)$
for all $S\subseteq E$, because constant variables always have entropy
$0$, and the extended polymatroid is still entropic. Now given $\varepsilon>0$
and $S\subseteq E$, the finite type hypothesis implies that for all
large enough $n$ we have $\left|h\left(S\cap E_{n}\right)-h\left(S\right)\right|<\varepsilon$.
Therefore 
\[
\left|\hat{h}_{n}\left(S\right)-h\left(S\right)\right|=\left|h_{n}\left(S\cap E_{n}\right)-h\left(S\right)\right|\le
\]
\[
\left|h_{n}\left(S\cap E_{n}\right)-h\left(S\cap E_{n}\right)\right|+\left|h\left(S\cap E_{n}\right)-h\left(S\right)\right|<\varepsilon+\frac{1}{n}
\]
and the claim follows.
\end{proof}
\begin{defn}
Let $n\in\mathbb{N}$ and identify each entropic polymatroid $\mathcal{P}\left(\left\{ 1,\ldots,n\right\} \right)\rightarrow\mathbb{R}$
with the corresponding vector in $\mathbb{R}^{2^{n}}$. We denote
the set of all such entropic polymatroids by $\Gamma_{n}^{*}\subset\mathbb{R}^{2^{n}}$.
Its closure is $\overline{\Gamma_{n}^{*}}$, the \emph{closed entropic
cone}. This is a convex cone: see \cite{Yeung_network_coding}.
\end{defn}

\section{Some almost entropic geometry}

\label[section]{sec:aent_geometry}

In projective geometry over a field, any pair of coplanar lines intersects
at a unique point. Various incidence theorems hold, most importantly
the theorems of Desargues and Pappus\footnote{Abstract projective geometries satisfying these two theorems actually
come from a field: in a sense that can be made precise, Desargues'
and Pappus' theorems are equivalent to associativity and commutativity
of multiplication, respectively.}. Our setting is not as rich: even pairs of coplanar lines need not
intersect. But it does satisfy several incidence theorems, including
Desargues'. It is useful to clarify what we mean by an incidence theorem
in the language of polymatroids; the following is a classical example.
\begin{example}
Let $k$ be a field. Define a matroid $r_{k}:\mathcal{P}\left(\mathbb{P}_{k}^{2}\right)\rightarrow\mathbb{R}$
on the projective plane over $k$ by setting 
\[
r_{k}\left(S\right)=\begin{cases}
0 & S=\emptyset\\
1 & S\text{ is a singleton}\\
2 & \left|S\right|\ge2\text{ is contained in a line}\\
3 & \text{otherwise.}
\end{cases}
\]
This is the rank function associated to the projective plane, considered
as a combinatorial geometry. Let $\mathcal{C}$ be the class of all
restrictions of matroids of the form $r_{k}$ (where $k$ is a field). 

Let $a,b,c,d\in\mathbb{P}_{k}^{2}$ be four distinct points, no three
on a line, and denote $E=\left\{ a,b,c,d\right\} $. This determines
$f\coloneqq r_{k}\restriction_{\mathcal{P}\left(E\right)}$ uniquely
(it is the uniform matroid of rank $3$ on $4$ points). Each of the
two pairs $\left\{ a,b\right\} $ and $\left\{ c,d\right\} $ spans
a line in $\mathbb{P}_{k}^{2}$, and these lines intersect at a unique
point $e$. Denoting $\hat{E}=E\cup\{e\}$ and $\hat{f}=r_{k}\restriction_{\mathcal{P}\left(\hat{E}\right)}$,
we observe that $\hat{f}$ is independent of $k$ and of the choice
of $a,b,c,d$. The following ``incidence theorem'' expresses that
fact that projective lines intersect nontrivially: given $\left(T,g\right)\in\mathcal{C}$
and an embedding $\eta:\left(E,f\right)\rightarrow\left(T,g\right)$,
there is a $\left(\hat{T},\hat{g}\right)\in\mathcal{C}$ and a commutative
diagram of polymatroid embeddings:
\[
\xymatrix{\left(\hat{E},\hat{f}\right)\ar[r] & \left(\hat{T},\hat{g}\right)\\
\left(E,f\right)\ar[u]\ar[r]^{\eta} & \left(T,g\right)\ar[u]
}
\]
where the embedding $\left(E,f\right)\rightarrow\left(\hat{E},\hat{f}\right)$
is given by the inclusion $E\subset\hat{E}$. This expresses the fact
that any two lines of each $\mathbb{P}_{k}^{2}$ intersect at a point.
\end{example}

The incidence theorem of the preceding example does not hold for the
class of almost entropic polymatroids\footnote{In fact, it does not even hold for the class of algebraic matroids,
which is known to be smaller by \cite{Matus_algebraic_aent}.}. However, there is a related theorem with stronger assumptions that
does hold, which we call the \emph{three line intersection theorem}.
This means that three lines, pairwise coplanar but not all lying in
the same plane, have a common point of intersection (this also holds
in projective spaces over a field). Formally:
\begin{thm}
Let $E=\left\{ a_{1},a_{2},b_{1},b_{2},c_{1},c_{2}\right\} $. Let
$f:\mathcal{P}\left(E\right)\rightarrow\mathbb{R}$ be the matroid
with circuits $\left\{ a_{1},a_{2},b_{1},b_{2}\right\} $, $\left\{ a_{1},a_{2},c_{1},c_{2}\right\} $,
and $\left\{ b_{1},b_{2},c_{1},c_{2}\right\} $, so that in particular:
\[
f\left(E\right)=4,\quad f\left(\{x\}\right)=1\text{ for each singleton \ensuremath{x},}
\]
\[
f\left(C\right)=3\text{ for each of the three circuits,}\quad f\left(a_{1},a_{2}\right)=f\left(b_{1},b_{2}\right)=f\left(c_{1},c_{2}\right)=2.
\]
Let $\hat{E}=E\cup\left\{ d\right\} $, and let $\hat{f}:\mathcal{P}\left(\hat{E}\right)\rightarrow\mathbb{R}$
be the matroid with the six circuits $\left\{ a_{1},a_{2},b_{1},b_{2}\right\} $,
$\left\{ a_{1},a_{2},c_{1},c_{2}\right\} $, $\left\{ b_{1},b_{2},c_{1},c_{2}\right\} $,
$\left\{ a_{1},a_{2},d\right\} $, $\left\{ b_{1},b_{2},d\right\} $,
$\left\{ c_{1},c_{2},d\right\} $. Note that the inclusion $E\subset\hat{E}$
induces an embedding $i:\left(E,f\right)\rightarrow\left(\hat{E},\hat{f}\right)$. 

If $\eta:\left(E,f\right)\rightarrow\left(T,g\right)$ is an embedding
with $\left(T,g\right)$ almost entropic, there exists an embedding
of $\left(T,g\right)$ into an almost entropic polymatroid $\left(\hat{T},\hat{g}\right)$
such that the diagram
\[
\xymatrix{\left(\hat{E},\hat{f}\right)\ar[r] & \left(\hat{T},\hat{g}\right)\\
\left(E,f\right)\ar[r]^{\eta}\ar[u]^{i} & \left(T,g\right)\ar[u]
}
\]
commutes.
\end{thm}

For $\left(T,g\right)$ a positive multiple of an entropic polymatroid\footnote{A positive multiple of an entropic polymatroid is a polymatroid of
the form $\left(T,c\cdot g\right)$ where $c>0$ and $\left(T,g\right)$
is entropic.} this can be shown directly, and the additional random variable corresponding
to $d$ can be constructed in the same sample space as the one corresponding
to $\left(T,g\right)$. For almost entropic polymatroids this is more
subtle: the result is a consequence of the Ahlswede-K\"orner lemma
\cite[Lemma 5]{Makarychev_Makarychev_Romashchenko_Vereshchagin}.
I first learned the almost entropic case in a conversation with Janneke
Bolt, Andrei Romashchenko, and Alexander Shen; since then a proof
has also appeared in \cite[Prop. 3.16]{Bamiloshon_Farras_Padro_extension_properties}.

\subsection{The copy lemma}

The copy lemma (see \cite{DFZ1, Matus_adhesivity, DFZ2}) is a very
useful tool for working with (almost) entropic polymatroids. It is
a primary tool in the search for nontrivial constraints on almost
entropic polymatroids. In this paper we use it to deduce incidence
theorems: on its own, the three-line intersection theorem says nothing
about planar (i.e. rank $3$) configurations, because its input data
is a restriction of rank $4$. The copy lemma allows us to lift subconfigurations
into a higher-rank space. A form of the copy lemma for almost entropic
polymatroids follows:
\begin{lem}
\label[lemma]{lem:copy}Let $\left(E,f\right)$ be an almost entropic
polymatroid, and let $B,C\subseteq E$ be disjoint subsets ($B$ is
the ``base'' for the operation, and $C$ is the set to be copied).
Let $C'$ be a disjoint copy of the set $C$: $C'=\left\{ c'\mid c\in C\right\} $.
Then there exists an almost entropic polymatroid $\left(\tilde{E},\tilde{f}\right)$,
where $\tilde{E}=E\sqcup C'$, satisfying:
\begin{enumerate}
\item For all $S\subseteq E$: $\tilde{f}\left(S\right)=f\left(S\right)$.
\item If $S\subseteq B\cup C$ then $f\left(S\right)=\tilde{f}\left(\left(S\cap B\right)\cup\left\{ c'\mid c\in S\right\} \right)$
(so $\tilde{f}\restriction_{B\cup C}$ is isomorphic to $\tilde{f}\restriction_{B\cup C'}$
via the ``obvious'' bijection).
\item For all $S\subseteq\tilde{E}$: $\tilde{f}_{/B}\left(S\right)=\tilde{f}_{/B}\left(S\cap C'\right)+\tilde{f}_{/B}\left(S\setminus C'\right)$
($C'$ is independent from $E\setminus C'$ over $B$).
\end{enumerate}
\end{lem}

\begin{rem}
These constraints do not fully determine $\tilde{f}$ in general,
but they do so in some cases of interest. In other cases, (in the
notation of the lemma,) the same almost entropic polymatroid $\left(E,f\right)$
may give rise to different $\left(\tilde{E},\tilde{f}\right)$ depending
on the approximating sequence of entropic polymatroids and choices
made in the construction.
\end{rem}

\subsection{Desargues' theorem for almost entropic polymatroids}

The classical theorem of Desargues states that if $O,a_{1},a_{2},a_{3},b_{1},b_{2},b_{3},x_{1},x_{2}$
are points in a projective space over a field satisfying:
\[
\begin{cases}
O,a_{i},b_{i}\text{ are collinear for each \ensuremath{i}},\\
a_{2},a_{3},x_{1}\text{ and }b_{2},b_{3},x_{1}\text{ are collinear,}\\
a_{1},a_{3},x_{2}\text{ and }b_{1},b_{3},x_{2}\text{ are collinear},\\
a_{1},a_{2},a_{3}\text{ are not collinear},
\end{cases}
\]
then the point $x_{3}$ at the intersection of the lines spanned by
$a_{1},a_{2}$ and by $b_{1},b_{2}$ satisfies that $x_{1},x_{2},x_{3}$
are collinear. 

In polymatroidal language, a set of elements is collinear if it has
rank $2$, and the line spanned by a pair of independent elements
is their closure. With this interpretation the analogous statement
holds almost verbatim for almost entropic polymatroids, except that
it is also a nontrivial claim that the point $x_{3}$ exists (in some
extension). We only need the rank-4 case, in which $O,a_{1},a_{2},a_{3}$
are not coplanar.
\begin{thm}
\label[theorem]{thm:aent_Desargues}Let $\left(E,f\right)$ be an
almost entropic polymatroid. Assume the subset
\[
C=\left\{ O,a_{1},a_{2},a_{3},b_{1},b_{2},b_{3},x_{1},x_{2}\right\} \subseteq E
\]
satisfies:
\[
\begin{cases}
f(p)=1\text{ for each }p\in C,\\
f(p,q)=2\text{ for distinct }p,q\in C,\\
f\left(O,a_{i},b_{i}\right)=2\text{ for each \ensuremath{1\le i\le3}},\\
f\left(a_{2},a_{3},x_{1}\right)=2=f\left(b_{2},b_{3},x_{1}\right),\\
f\left(a_{1},a_{3},x_{2}\right)=2=f\left(b_{1},b_{3},x_{2}\right),\\
f\left(O,a_{1},a_{2},a_{3}\right)=4. & \text{(This is the rank \ensuremath{4} assumption.)}
\end{cases}
\]
Then there exists an extension $\left(\hat{E},\hat{f}\right)$ of
$\left(E,f\right)$ and $x_{3}\in\hat{E}$ such that in addition we
have
\[
\begin{cases}
\hat{f}\left(x_{3}\right)=1,\\
\hat{f}\left(a_{1},a_{2},x_{3}\right)=2=\hat{f}\left(b_{1},b_{2},x_{3}\right),\\
\hat{f}\left(a_{1},x_{3}\right)=2=\hat{f}\left(b_{1},x_{3}\right),\\
\hat{f}\left(x_{1},x_{2},x_{3}\right)=2.
\end{cases}
\]
Further, if there exists $\tilde{x}_{3}\in E\subset\hat{E}$ such
that $f\left(a_{1},a_{2},\tilde{x}_{3}\right)=2=f\left(b_{1},b_{2},\tilde{x}_{3}\right)$,
$f\left(\tilde{x}_{3}\right)=1$, and $f\left(\tilde{x}_{3},a_{1}\right)=2=f\left(\tilde{x}_{3},b_{1}\right)$,
then $\hat{f}\left(x_{3},\tilde{x}_{3}\right)=1$ and $f\left(x_{1},x_{2},\tilde{x}_{3}\right)=2$.
\end{thm}

While this is essentially known for entropic polymatroids (by completely
different methods, see \cite{Matus_partitions}), I have not been
able to find a reference for the almost entropic case in the literature.
\begin{rem}
The rank $4$ assumption can be replaced by the assumption that $f\left(a_{1},a_{2},a_{3}\right)=3$
and $f\restriction_{\mathcal{P}\left(C\right)}$ is a matroid. In
this case one applies the copy lemma in order to lift the configuration
to a higher dimensional space, applies the rank $4$ case, and ``projects''
back. The proof is entirely analogous to the usual projective one.
\end{rem}

\begin{proof}
It suffices to prove that the six points $\left\{ a_{1},a_{2},b_{1},b_{2},x_{1},x_{2}\right\} $
satisfy the assumptions of the 3-line intersection lemma (with $x_{1},x_{2}$
playing the role of $c_{1},c_{2}$) because then the resulting intersection
point can be taken to be $x_{3}$.

We first prove that $f\left(a_{1},a_{2},b_{1}\right)=3$: since $f\left(O,a_{1},b_{1}\right)=2=f\left(a_{1},b_{1}\right)$
we have $O\in\mathrm{cl}\left(a_{1},a_{2},b_{1}\right)$, and similarly
$b_{1}\in\mathrm{cl}\left(O,a_{1},a_{2}\right)$. Hence
\[
f\left(a_{1},a_{2},b_{1}\right)=f\left(a_{1},a_{2},O,b_{1}\right)=f\left(a_{1},a_{2},O\right)=3
\]
where the last equality holds because $f\left(O,a_{1},a_{2},a_{3}\right)=4$
and each of the singletons has rank $1$, so they are independent.
In the same way also $f\left(a_{1},a_{2},b_{2}\right)=3$.

Also $f\left(a_{1},a_{2},x_{1}\right)=3$, because $f\left(a_{2},a_{3},x_{1}\right)=2=f\left(a_{2},x_{1}\right)=f\left(a_{2},a_{3}\right)$,
so $x_{1}\in\mathrm{cl}\left(a_{2},a_{3}\right)$ and $a_{3}\in\mathrm{cl}\left(a_{2},x_{1}\right)$.
The wanted equality follows as above. Similarly $f\left(a_{1},a_{2},x_{2}\right)=3$.

Now consider $f\left(a_{1},a_{2},b_{i},x_{j}\right)$ for $1\le i,j\le2$:
this rank is at most $4$ (the sum of the ranks of the singletons)
and we have $a_{3}\in\mathrm{cl}\left(a_{1},a_{2},x_{j}\right)$ (for
instance, if $j=1$ we have $a_{3}\in\mathrm{cl}\left(a_{2},x_{1}\right)$).
We also have $O\in\mathrm{cl}\left(a_{i},b_{i}\right)$, so that 
\[
f\left(a_{1},a_{2},b_{i},x_{j}\right)\ge f\left(O,a_{1},a_{2},a_{3}\right)=4.
\]
Hence the rank equals precisely $4$.

Also $f\left(x_{1},x_{2},a_{i},b_{j}\right)=4$ for $1\le i,j\le2$:
$4$ is clearly an upper bound (the sum of ranks of the singletons).
Assume without loss of generality that $i=1$. Then $a_{3}\in\mathrm{cl}\left(x_{2},a_{i}\right)$.
Also, $a_{2}\in\mathrm{cl}\left(x_{1},a_{3}\right)\subset\mathrm{cl}\left(x_{1},x_{2},a_{i}\right)$.
Hence
\[
f\left(x_{1},x_{2},a_{i},b_{j}\right)\ge f\left(a_{1},a_{2},a_{3},b_{j}\right),
\]
where $O\in\mathrm{cl}\left(a_{j},b_{j}\right)$. So the last rank
is at least $f\left(O,a_{1},a_{2},a_{3}\right)=4$.

Observe that $f\left(O,b_{1},b_{2},b_{3}\right)=4$ since the closure
of this set contains $\left\{ O,a_{1},a_{2},a_{3}\right\} $ (and
$4$ is again an upper bound on the rank). In particular, the above
equalities also hold on exchanging $\left\{ a_{1},a_{2}\right\} $
with $\left\{ b_{1},b_{2}\right\} $. Up to performing such an exchange,
every subset of cardinality $3$ of $\left\{ a_{1},a_{2},b_{1},b_{2},x_{1},x_{2}\right\} $
is a subset of one of the sets considered above, and hence has rank
$3$. We have also shown that every subset of cardinality $4$ which
does not contain two of the pairs $\left\{ a_{1},a_{2}\right\} $,
$\left\{ b_{1},b_{2}\right\} $, and $\left\{ x_{1},x_{2}\right\} $
has rank $4$ (and the pairs of pairs have rank $3$ by assumption).
Since the entire polymatroid has rank $4$, this determines it completely
and proves that it equals the matroid of the three-line intersection
theorem.

Now assume there exists $\tilde{x}_{3}\in E\subset\hat{E}$ with $f\left(a_{1},a_{2},\tilde{x}_{3}\right)=2=f\left(b_{1},b_{2},\tilde{x}_{3}\right)$,
$f\left(\tilde{x}_{3}\right)=1$, and $f\left(\tilde{x}_{3},a_{1}\right)=2=f\left(\tilde{x}_{3},b_{1}\right)$.
Then $\mathrm{cl}\left(\tilde{x}_{3},a_{1}\right)\ni a_{2}$ and since
$\mathrm{cl}^{2}=\mathrm{cl}$ we obtain $\mathrm{cl}\left(\tilde{x}_{3},a_{1}\right)=\mathrm{cl}\left(\tilde{x}_{3},a_{1},a_{2}\right)\ni x_{3}$.
This implies 
\[
f\left(\tilde{x}_{3},x_{3},a_{1}\right)=f\left(\tilde{x}_{3},x_{3},a_{1},a_{2}\right)=2,
\]
and in the same way also $f\left(\tilde{x}_{3},x_{3},b_{1}\right)=2$.
We also have $f\left(\tilde{x}_{3},x_{3},a_{1},b_{1}\right)=3$: on
the one hand, by monotonicity, it is at least $f\left(x_{3},a_{1},b_{2}\right)=3$
(this is part of the definition of the three-line intersection configuration).
On the other hand, this set is contained in the closure of $\left\{ a_{1},a_{2},b_{1}\right\} $,
which has rank $3$, (again by definition of the three-line intersection
configuration,) and monotonicity gives the wanted equality. Now $f_{/\tilde{x}_{3}}\left(x_{3},b_{1}\right)=1=f_{/\tilde{x}_{3}}\left(x_{3},a_{1}\right)$
and $f_{/\tilde{x_{3}}}\left(x_{3},a_{1},b_{1}\right)=2$ by definition.
By submodularity,
\[
f_{/\tilde{x}_{3}}\left(\left\{ x_{3},a_{1}\right\} \cap\left\{ x_{3},b_{1}\right\} \right)+f_{/\tilde{x}_{3}}\left(\left\{ x_{3},a_{1}\right\} \cup\left\{ x_{3},b_{1}\right\} \right)\le f_{/\tilde{x}_{3}}\left(x_{3},a_{1}\right)+f_{/\tilde{x}_{3}}\left(x_{3},b_{1}\right)
\]
\[
f_{/\tilde{x}_{3}}\left(x_{3}\right)+f_{/\tilde{x}_{3}}\left(x_{3},a_{1},b_{1}\right)\le2
\]
and hence $f_{/\tilde{x}_{3}}\left(x_{3}\right)=f\left(\tilde{x}_{3},x_{3}\right)-f\left(\tilde{x}_{3}\right)=0$,
so $f\left(\tilde{x_{3}},x_{3}\right)=1$ as required. We have $x_{3}\in\mathrm{cl}_{\hat{f}}\left(\tilde{x}_{3}\right)$
and $\tilde{x_{3}}\in\mathrm{cl}_{\hat{f}}\left(x_{3}\right)$, so
that 
\[
\hat{f}\left(x_{1},x_{2},x_{3}\right)=\hat{f}\left(x_{1},x_{2},x_{3},\tilde{x}_{3}\right)=\hat{f}\left(x_{1},x_{2},\tilde{x}_{3}\right)=2.
\]
Since $\left(\hat{E},\hat{f}\right)$ extends $\left(E,f\right)$,
we obtain $f\left(x_{1},x_{2},\tilde{x}_{3}\right)=2$.
\end{proof}

\subsubsection*{Notation for Desargues configurations}
\begin{defn}
Let $\left(E,f\right)$ be a polymatroid. Triples $\left(a_{1},a_{2},a_{3}\right)$
and $\left(b_{1},b_{2},b_{3}\right)$ of elements of $E$ are \emph{in
perspective} from $O\in E$ if $f\left(O,a_{i},b_{i}\right)=2$ for
all $1\le i\le3$.
\end{defn}

Thus a Desargues configuration in $\left(E,f\right)$ consists of: 
\begin{itemize}
\item a point $O\in E$, 
\item two triples $\left(a_{1},a_{2},a_{3}\right)$ and $\left(b_{1},b_{2},b_{3}\right)$
in perspective from $O$, and 
\item a pair of ``intersection points'' $x_{1}\in\mathrm{cl}\left(a_{2},a_{3}\right)\cap\mathrm{cl}\left(b_{2},b_{3}\right)$
and $x_{2}\in\mathrm{cl}\left(a_{1},a_{3}\right)\cap\mathrm{cl}\left(b_{1},b_{3}\right)$,
\end{itemize}
such that each of the points has rank $1$, each pair has rank $2$,
and the subset $\left\{ O,a_{1},a_{2},a_{3}\right\} $ has rank $4$.

\section{Partial Dowling geometries}

\label[section]{sec:PDG}

Partial Dowling geometries are a class of matroids that mildly generalizes
the Dowling geometries (\cite{Dowling_geometries}). L. K\"{u}hne
and I introduced them in \cite{KY19} in a context similar to that
of the current paper. One constructs PDGs from certain group presentations
$\left\langle S\mid R\right\rangle $ (with various assumptions on
the presentation, which are easy to obtain in practice and do not
constrain the resulting group). A representation of a partial Dowling
geometry in each of the multilinear, entropic, and almost-multilinear
settings gives rise to a certain representation of the group, and
in particular witnesses that the presentation's generators (other
than the trivial element, which we assume to be in $S$) are nontrivial
(\cite{KY19, KY22}). A main result of this paper is that the same
holds for almost entropic representations, and more generally for
classes of representations that admit a copy lemma and a three-line
intersection theorem.

It is useful to generalize the definition of \cite{KY19} to allow
polymatroids rather than just matroids, and rank larger than $3$:
some of our constructions enlarge a given PDG to one in which the
rank function is no longer integral, and we work with PDGs of rank
$4$. Simple matroids of rank $3$ which are PDGs in the sense given
here are precisely the PDGs of \cite{KY19,KY22}.
\begin{defn}
A partial Dowling geometry (PDG) of rank $r$ is a triple $\left(E,f,S\right)$
such that $\left(E,f\right)$ is a polymatroid and $S$ is a set,
called the \emph{generating set}, such that the following hold:
\begin{enumerate}
\item \label{PDG:S}$S$ is equipped with an involution $\left(\cdot\right)^{-1}$
and has a distinguished element $e$ satisfying $e^{-1}=e$.
\item \label{PDG:B}$E$ has a subset $B=\left\{ b_{1},\ldots,b_{r}\right\} $,
called the distinguished basis, satisfying that for all $A\subseteq B$:
$f\left(A\right)=\left|A\right|$.
\item \label{PDG:E}For each $1\le i<j\le r$, $E$ contains a copy of $S$
with elements denoted $\left\{ s_{i,j}\mid s\in S\right\} $; these
${r \choose 2}$ copies are disjoint from each other and from $B$,
so that 
\[
E=B\sqcup\bigsqcup_{1\le i<j\le r}\left\{ s_{i,j}\mid s\in S\right\} .
\]
For $i<j$ we extend the involution to $\left\{ s_{i,j}\mid s\in S\right\} $
by setting $s_{i,j}^{-1}\coloneqq\left(s^{-1}\right)_{i,j}$. For
$i<j$ we denote $s_{j,i}\coloneqq s_{i,j}^{-1}$ and $s_{j,i}^{-1}\coloneqq s_{i,j}$.
\item \label{PDG:frame}For all $1\le i,j\le r$ we have $s_{i,j}\in\mathrm{cl}\left(b_{i},b_{j}\right)$.
\item \label{PDG:nondegeneracy}For each $s\in S$ and each $1\le i,j\le r$
we have $f\left(s_{i,j}\right)=1$ and $f\left(b_{i},s_{i,j}\right)=2$.
\item \label{PDG:coherence}For each $s\in S$ and all $1\le i,j,k\le r$
we have $f\left(s_{i,j},s_{j,k}^{-1},e_{k,i}\right)=2$.
\end{enumerate}
A \emph{relator} of a PDG is a triple $\left(s,s',s''\right)$ of
elements of $S$ such that $f\left(s_{i,j},s_{j,k}^{\prime},s_{k,i}^{\prime\prime}\right)=2$
whenever $1\le i,j,k\le r$ are distinct.
\begin{defn}
An \emph{incoherent PDG} is a triple $\left(E,f,S\right)$ which satisfies
all conditions for PDGs except possibly (\ref{PDG:E},\ref{PDG:coherence}).
Instead of (\ref{PDG:E}), we require only that $B\subseteq E\subseteq B\sqcup\bigsqcup_{1\le i<j\le r}\left\{ s_{i,j}\mid s\in S\right\} $;
the conditions (\ref{PDG:frame},\ref{PDG:nondegeneracy}) are then
assumed to hold only for those elements $s_{i,j}$ that are in $E$.

Note that in this case we still have $s_{j,i}^{-1}=s_{i,j}$ whenever
$s_{i,j}\in E$. 
\end{defn}

\end{defn}

\begin{rem}
\begin{enumerate}
\item PDGs are coherent (i.e. do satisfy (\ref{PDG:E},\ref{PDG:coherence}))
unless we explicitly state otherwise. The notion of incoherent PDGs
is useful mainly because it allows us to apply some lemmas to partially
constructed PDGs, in order to prove that they satisfy the conditions.
\item Condition (\ref{PDG:frame}) implies that $f\left(\left\{ b_{i},b_{j}\right\} \cup\left\{ s_{i,j}\mid s\in S\right\} \right)=2$.
In the presence of condition (\ref{PDG:B}), these statements are
equivalent.
\item Condition (\ref{PDG:frame}) also implies that $\mathrm{cl}\left(B\right)=E$
and hence $f\left(E\right)=r$. 
\item Condition (\ref{PDG:coherence}) is equivalent to the statement that
$\left(s,s^{-1},e\right)$ is a relator of $\left(E,f,S\right)$ for
each $s\in S$. In particular $\left(e,e,e\right)$ is always a relator.
\item Like in groups, relators are closed under cyclic shifts and inversion
(of the order as well as the individual generators): if $\left(s,s',s''\right)$
is a relator then so are $\left(s',s'',s\right)$ and $\left(s''{}^{-1},s'{}^{-1},s^{-1}\right)$.
\item In previous papers we used circuits rather than working with the rank
directly; this is more convenient when possible, but unsuitable for
working with polymatroids.
\end{enumerate}
\end{rem}

\subsubsection*{The PDG of a finitely presented group}

The following construction is a straightforward adaptation from \cite{KY19},
and always yields a matroid.

\begin{construction}[PDGs from group presentations]\label[construction]{cons:fp_group_PDG}
Let $G=\left\langle S\mid R\right\rangle $ be a finitely presented
group. Assume the generating set is symmetric (closed under $\left(\cdot\right)^{-1}$),
all relations are of length $3$ (of form $ss's''=e_{G}$), and that
$e_{G}\in S$.\footnote{For the resulting polymatroid to be almost entropic, an additional
necessary (but far from sufficient) condition is that the relations
are symmetric, in the sense that if $ss's''=e_{G}$ is a relation
then so is $s''ss'=e_{G}$ (this additional relation holds: the left
hand side is $s''\cdot\left(ss's''\right)\cdot s''{}^{-1}$). We call
such presentations \emph{symmetric triangular.}} We construct a partial Dowling geometry $\left(E,f,S\right)$ of
rank $3$, with generating set $S$ exactly the given generating set
of $G$.

The involution is precisely the inversion map on $S$, and the only
data left to be specified is the rank function (insofar as it is not
completely determined by the definition of PDGs). 

For any $T\subseteq E$ of cardinality $1$ or $2$ we set $f\left(T\right)=\left|T\right|$.
If $\left|T\right|=3$, we set $f\left(T\right)=2$ if $T\subseteq\left\{ b_{i},b_{j}\right\} \cup\left\{ s_{i,j}\mid s\in S\right\} $
or if $T=\left\{ s_{i,j},s'_{j,k},s''_{k,i}\right\} $ where $ss's''=e_{G}$
is a relation in $R$; otherwise we set $f\left(T\right)=3$. For
$\left|T\right|\ge3$ we set 
\[
f\left(T\right)=\max_{T'\subset T,\left|T'\right|=3}f\left(T'\right).
\]

\end{construction}

Note that the relators of the PDG of $\left\langle S\mid R\right\rangle $
are exactly those in $R$.

\subsubsection*{Nondegeneracy}

It is useful to know that certain subsets of a PDG are independent.
In many cases this is impossible to guarantee: for example, if $s,s'\in S$
are distinct, it is entirely possible that $f\left(s_{1,2},s_{1,2}^{\prime}\right)=\frac{3}{2}$.
However, it is not difficult to prove that $f\left(s_{1,2},s_{2,3}^{\prime}\right)=2$
always holds. Since we need this sort of statement in several cases,
we give a relatively general independence criterion for subsets disjoint
from the distinguished basis.

In the following $K_{n}$ denotes the complete graph on $\left[n\right]=\left\{ 1,\ldots,n\right\} $.
Its edge set is ${\left[n\right] \choose 2}$, the set of unordered
pairs of elements of $\left[n\right]$.
\begin{lem}
\label[lemma]{lem:acyclic_independence}Let $\left(E,f,S\right)$
be a (possibly incoherent) PDG of rank $r$. Let $s^{(1)},\ldots,s^{(n)}\in S$
(not necessarily distinct) and for each $1\le t\le n$ let $1\le i_{t},j_{t}\le r$
be a distinct pair of indices such that $\left\{ i_{t},j_{t}\right\} \neq\left\{ i_{u},j_{u}\right\} $
whenever $t\neq u$. Let $A=\left\{ s_{i_{t},j_{t}}^{(t)}\right\} _{1\le t\le n}$,
and assume that $s_{i_{t},j_{t}}^{(t)}\in E$ for all $t$. If the
subgraph $G$ of $K_{r}$ with edges $E(G)=\left\{ \left\{ i_{t},j_{t}\right\} \mid1\le t\le n\right\} $
and vertices $V(G)=\bigcup_{t}\left\{ i_{t},j_{t}\right\} $ contains
no cycle then $A$ has rank $n$ in $\left(E,f\right)$.
\end{lem}

\begin{proof}
First assume $G$ is connected. Then it is a tree with $n$ edges,
and hence has $n+1$ vertices. Define $Z=\mathrm{cl}\left(A\cup\left\{ b_{i_{1}}\right\} \right)$.
It suffices to prove $Z\supseteq\left\{ b_{i}\mid i\in V(G)\right\} $,
because then 
\[
f\left(A\right)+f\left(b_{i_{1}}\right)\ge f\left(A\cup\left\{ b_{i_{1}}\right\} \right)\ge\left|V(G)\right|=n+1
\]
(where the first inequality holds by subadditivity). Since $f\left(b_{i_{1}}\right)=1$,
this yields $f(A)=n$.

To see this, observe that for any $s\in S$ and any $1\le i,j\le r$,
if $s_{i,j}\in E$ then $\mathrm{cl}\left(b_{i},s_{i,j}\right)\ni b_{j}$:
we have $f\left(b_{i},s_{i,j},b_{j}\right)=2=f\left(b_{i},s_{i,j}\right)$
(by conditions (\ref{PDG:frame},\ref{PDG:nondegeneracy}) of the
definition). In the same way, $\mathrm{cl}\left(b_{j},s_{i,j}\right)\ni b_{i}$.
Now we use a connectivity argument: let $W=\left\{ 1\le i\le r\mid b_{i}\in Z\right\} $.
If $W\subsetneq V(G)$ is a proper subset, there is some edge $\left\{ i_{t},j_{t}\right\} \in E(G)$
with precisely one endpoint in $W$, say $i_{t}$. Then $b_{i_{t}}\in Z$
and $s_{i_{t},j_{t}}\in Z$, so also $b_{j_{t}}\in\mathrm{cl}\left(Z\right)=Z$
(since $\mathrm{cl}^{2}=\mathrm{cl}$) and hence $j_{t}\in W$, contradicting
the choice of $\left\{ i_{t},j_{t}\right\} $. Hence $W=V(G)$. This
proves the connected case.

If $G$ is disconnected, it is a forest and adding some finite set
of edges $\left\{ i_{n+1},j_{n+1}\right\} $, ..., $\left\{ i_{n+m},j_{n+m}\right\} $
makes it a tree. The argument for the connected case applies to 
\[
A\cup\left\{ e_{i_{n+1},j_{n+1}},\ldots,e_{i_{n+m},j_{n+m}}\right\} ,
\]
which must have rank $n+m$ and is hence independent. So $f(A)=\left|A\right|$.
\end{proof}
The next lemma is rather specific, but together with the preceding
one it covers all cases we need.
\begin{lem}
\label[lemma]{lem:weak_pair_independence}Let $\left(E,f,S\right)$
be a (possibly incoherent) PDG. Let $s\in S$, let $1\le i\le r$,
and let $1\le j,k\le r$ (not necessarily distinct from $i$). If
$s_{j,k}\in E$ then $f\left(b_{i},s_{j,k}\right)=2$.
\end{lem}

\begin{proof}
If $j=i$ or $k=i$ this is condition (\ref{PDG:nondegeneracy}) of
the definition. Otherwise, $\mathrm{cl}\left(b_{i},b_{j},s_{j,k}\right)\supseteq\left\{ b_{i},b_{j},b_{k}\right\} $,
so $f\left(b_{i},b_{j},s_{j,k}\right)=3$, and hence $f\left(b_{i},s_{j,k}\right)=2$
(by subadditivity). 
\end{proof}
\begin{cor}
\label[corollary]{cor:pair_independence}Let $\left(E,f,S\right)$
be a (possibly incoherent) PDG. Any pair $\left\{ x,y\right\} $ of
elements of $E$ have $f\left(x,y\right)=2$ unless $x=s_{i,j}$ and
$y=s_{i,j}^{\prime}$ for some $s,s'\in S$ and $1\le i,j\le r$.
\end{cor}

\begin{proof}
If at least one of $x,y$ is a distinguished basis element (i.e. of
the form $b_{i}$) this follows from condition (\ref{PDG:B}) of the
definition and from \cref{lem:weak_pair_independence}. If $x$ and
$y$ are some $s_{i,j}$ and $s_{k,l}^{\prime}$, and $\left\{ i,j\right\} \neq\left\{ k,l\right\} $,
this follows from \cref{lem:acyclic_independence}. 
\end{proof}

\section{Products and associativity in partial Dowling geometries}

\label[section]{sec:PDG_groups}

Almost entropic partial Dowling geometries of rank $4$ carry algebraic
information: roughly speaking, and ignoring a technicality related
to parallel elements, if $\left(E,f,S\right)$ is a finite almost
entropic PDG, it has a (non-canonical, often infinite) extension $\left(\tilde{E},\tilde{f},G\right)$
with generating set $G\supseteq S$, satisfying that $G$ is a group
generated by $S$, and the multiplication operation can be read off
from the rank function: for each $g,g'\in G$ there is a unique relator
$\left(g,g',g''\right)$ of $\left(\tilde{E},\tilde{f},G\right)$,
and $g\cdot g'=g''{}^{-1}$ in $G$. If the original PDG is the one
associated to the group presentation $\left\langle S\mid R\right\rangle $
then $G$ is a quotient of $\left\langle S\mid R\right\rangle $,
and if $s,s'\in S$ are distinct then they map to distinct elements
of $G=\left\langle S\right\rangle $. (The precise statements are
close but not identical to this, see \cref{sec:group_from_geometric_product}.)

We begin by constructing a ``geometric product'' operation in \cref{sec:geometric_product_from_PDG}.
In \cref{sec:group_from_geometric_product} we show that the geometric
product is associative and use it to recover a group $G$.

\subsection{Desargues configurations and the geometric product}

\label[section]{sec:geometric_product_from_PDG}
\begin{defn}
Let $\left(E,f,S\right)$ be a PDG. A geometric product of $s,s'\in S$
is an element $t\in S$ such that $\left(s,s',t^{-1}\right)$ is a
relator of the PDG.

Note that such a $t\in S$ need not exist. If it exists it is nearly
unique, in the sense that $\mathrm{cl}(t)$ is precisely the set of
all geometric products: we prove this below in \cref{thm:PDG_prod_uniqueness}.
Also, a given finite almost entropic PDG of rank $4$ has an almost
entropic extension that is closed under the geometric product operation:
this is \cref{thm:prod_extension}.
\end{defn}

\begin{lem}
[A Desargues configuration in a PDG]\label[lemma]{lem:PDG_Desargues}
Let $\left(E,f,S\right)$ be a (possibly incoherent) PDG of rank $4$.
Let $\left(i,j,k,m\right)$ be a permutation of $\left(1,\ldots,4\right)$.
Let $s,u,w\in S$ and assume that there exist $t,v\in S$ such that
$\left(s^{-1},t,u\right)$ and $\left(u^{-1},v,w\right)$ are relators
of the PDG, or more generally (in the incoherent case) such that 
\[
w_{k,m},s_{i,m},u_{j,m},t_{i,j},v_{j,k}\in E
\]
and
\[
f\left(s_{m,i}^{-1},t_{i,j},u_{j,m}\right)=f\left(u_{m,j}^{-1},v_{j,k},w_{k,m}\right)=2.
\]
Then the triples $\left(b_{k},b_{i},b_{j}\right)$ and $\left(w_{k,m},s_{i,m},u_{j,m}\right)$
are in perspective from $b_{m}$, and they form a Desargues configuration
together with the intersection points $t_{i,j}\in\mathrm{cl}\left(s_{i,m},u_{j,m}\right)\cap\mathrm{cl}\left(b_{i},b_{j}\right)$
and $v_{j,k}\in$$\mathrm{cl}\left(u_{j,m},w_{k,m}\right)\cap\mathrm{cl}\left(b_{j},b_{k}\right)$.

In particular, if there exists $x\in S$ with $x_{k,i}\in E$ and
$f\left(w_{k,m},s_{i,m},x_{k,i}\right)=2$ then $f\left(t_{i,j},v_{j,k},x_{k,i}\right)=2$.
\end{lem}

\begin{notation}
As a visual shorthand, we depict the situation of the lemma as follows:

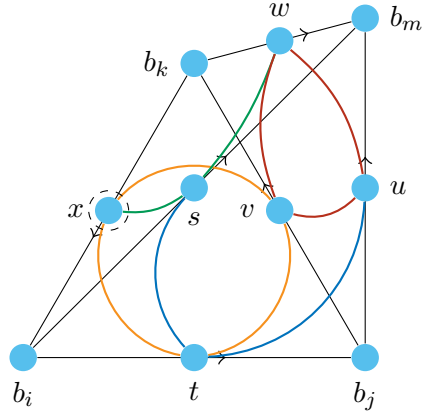
\begin{center}
\begin{tikzpicture}[scale=1.5,
	vert/.style = {
		circle,
		minimum size = 2mm,
		fill=Cerulean!60,
		font=\itshape
	},
    midarrow/.style={
    postaction={decorate},
    decoration={
      markings,
      mark=at position 0.6 with {\arrow[scale=1.5]{>}}
        }
    },
    midarrowShort/.style={
    postaction={decorate},
    decoration={
      markings,
      mark=at position 0.7 with {\arrow[scale=1.5]{>}}
        }
    },
    midarrowLong/.style={
    postaction={decorate},
    decoration={
      markings,
      mark=at position 0.57 with {\arrow[scale=1.5]{>}}
        }
    }
    ]
    \node (v1)  at (-3,0)		    [vert,label=below:$b_i$]{};
    \node (v2)  at (0,0)			[vert,label=below:$b_j$]{};

    \node (v3)  at (-1.5,2.6)		[vert,label=left:$b_k$]{};
    \node (v4)  at (0,3)            [vert,label=right:$b_m$]{};

    \draw[midarrow]         (v1) -- (v2);
    \draw[midarrow]         (v3) -- (v1);
    \draw[midarrow]         (v2) -- (v3);
    \draw[midarrow]         (v2) -- (v4);
    \draw[midarrow]         (v1) -- (v4);
    \draw[midarrowShort]    (v3) -- (v4);
    
    \node (m23) at (-0.75,1.3)      [vert,label=left:$v$]{};
    \node (m13) at (-2.25,1.3)      [vert,label=left:$x$]{};
    \node (m12) at (-1.5,0)		    [vert,label=below:$t$]{};
    
    \node (m24) at (0,1.5)          [vert,label=right:$u$]{};
    \node (m14) at (-1.5,1.5)       [vert,label=below:$s$]{};
    \node (m34) at (-0.75,2.8)      [vert,label=above:$w$]{};

    \draw[BurntOrange,thick]
        (m23) to [bend left=50] (m12) 
        (m12) to [bend left=50] (m13) 
        (m13) to [bend left=50] (m23);
    \draw[RoyalBlue,thick]
        (m12) to [bend left=42] (m14)
        (m24) to [bend left=42] (m12);
    \draw[ForestGreen,thick]
        (m14) to [bend right=10] (m34)
        (m13) to [bend right=20] (m14);
    \draw[BrickRed,thick]
        (m34) to [bend left=20] (m24)
        (m24) to [bend left=30] (m23)
        (m23) to [bend left=20] (m34);
    \draw[dashed] (m13) circle[radius=0.175];
\end{tikzpicture}
\end{center}

Here the curves\footnote{Combinatorially, this really is a rank $4$ Desargues configuration,
and it can be realized in $\mathbb{R}^{3}$ with straight lines and
the same point-line incidences. The cost is that we would have to
draw the points $t,v,x$ outside the tetrahedron (but still on the
lines that their edges currently span). It seems more convenient to
have a compact depiction that stresses the combinatorial symmetries.} represent combinatorial lines (sets of rank $2$). The assumptions
to be checked and the conclusions for Desargues' theorem can be read
off from the diagram; this is intended to make the assumptions easy
to check and the conclusion easy to see.

The label $s$ on the segment between $b_{i},b_{m}$ is to be interpreted
as $s_{i,m}$, because of the arrow on the segment pointing in the
direction from $i$ to $m$; in the same way, the label $w$ is to
be interpreted as $w_{k,m}$, and $x$ is to be interpreted as $x_{k,i}$.
Note that the directions of the arrows may vary in different diagrams;
if the arrow on the edge between $b_{i},b_{m}$ were flipped, we would
read the label $s$ as $s_{m,i}=s_{i,m}^{-1}$. The point $x$ has
an exceptional role, and hence is circled by an additional dashed
line (the lemma doesn't assume it exists; when we use the lemma without
such an $x$ in sight, we draw the dashed circle as empty to emphasize
this). The lemma states that given all points in the diagram except
potentially $x$, if the curves through the points labelled $\left\{ s,t,u\right\} $
and $\left\{ u,v,w\right\} $ have rank $2$, the diagram is a Desargues
configuration in perspective from $b_{m}$, the apex of the drawing.
\Cref{thm:aent_Desargues} assures us we can extend the polymatroid
to an almost entropic one that contains an $x$ as drawn; and if such
an $x$ already exists, then $f\left(t_{i,j},v_{j,k},x_{k,i}\right)=2$
provided that also the green curve through $\left\{ s,w,x\right\} $
has rank $2$.

In other words, when an element playing the role of $x$ already exists,
we need to make sure that three of the four colorful curves (all but
the orange one, through $\left\{ t,v,x\right\} $) have rank $2$,
and we obtain the same statement for the orange curve.
\end{notation}

\begin{proof}
The two triples are in perspective from $b_{m}$: corresponding points
are of the form $b_{p}$ and $z_{p,m}$ for some $z\in S$ and $1\le p\le3$,
and $f\left(b_{m},z_{p,m},b_{p}\right)=2$ by conditions (\ref{PDG:B},\ref{PDG:frame})
of the definition of PDGs. The union of the first triple with $b_{m}$
has rank $4$ by condition (\ref{PDG:B}), because it is precisely
$b_{1},\ldots,b_{4}$.

We have $t_{i,j}\in\mathrm{cl}\left(b_{i},b_{j}\right)$ and $v_{j,k}\in\mathrm{cl}\left(b_{j},b_{k}\right)$
by condition (\ref{PDG:frame}) of the definition. That $t_{i,j}\in\mathrm{cl}\left(s_{i,m},u_{j,m}\right)$
follows from the assumption that $\left(s^{-1},t,u\right)$ is a relator
(or, in the incoherent case, the weaker assumption that $f\left(s_{m,i}^{-1},t_{i,j},u_{j,m}\right)=2$),
and the fact that $s_{m,i}^{-1}=s_{i,m}$. Hence $f\left(s_{i,m},u_{j,m},t_{i,j}\right)=f\left(s_{i,m},u_{j,m}\right)$.
In the same way we have $v_{j,k}\in\mathrm{cl}\left(u_{j,m},w_{k,m}\right)$.

Each point has rank $1$ by condition (\ref{PDG:nondegeneracy}) of
the definition. It remains to check that each pair of points in the
configuration has rank $2$. This follows from \cref{cor:pair_independence},
since for each $1\le p,q\le4$ there is at most one $z\in S$ with
$z_{p,q}$ in the configuration.

This shows that the claimed elements form a Desargues configuration.
Now assume there exists $x\in S$ with $x_{k,i}\in E$ and $f\left(x_{k,i},w_{k,m},s_{i,m}\right)=2$.
We also have $f\left(x_{k,i},b_{k},b_{i}\right)=2$ by condition (\ref{PDG:frame})
of the definition of PDGs, $f\left(x_{k,i},w_{k,m}\right)=2=f\left(x_{k,i},b_{k}\right)$
by \cref{cor:pair_independence}, and $f\left(x_{k,i}\right)=1$ by
condition (\ref{PDG:nondegeneracy}). By the ``Further, ...'' clause
of \cref{thm:aent_Desargues} we obtain $f\left(t_{i,j},v_{j,k},x_{k,i}\right)=2$.
\end{proof}
\begin{lem}
[Finding relators]\label[lemma]{lem:finding_relators} Let $\left(E,f,S\right)$
be an almost entropic PDG of rank $4$ and let $s,s'\in S$. If there
exists $s''\in S$ and distinct indices $1\le p,q,r\le4$ such that
$f\left(s_{p,q},s_{q,r}^{\prime},s_{r,p}^{\prime\prime}\right)=2$
then $\left(s,s',s''\right)$ is a relator of $\left(E,f,S\right)$.
\end{lem}

\begin{proof}
Denote $i=p$, $m=q$, $k=r$, and let $1\le j\le4$ be the unique
index not in $\left(p,q,r\right)$. We apply \cref{lem:PDG_Desargues}
with $\left(s,u,w\right)=\left(s,e,s^{\prime-1}\right)$, $t=s$,
and $v=s'$ (note that $\left(s^{-1},t,u\right)=\left(s^{-1},s,e\right)$
and $\left(u^{-1},v,w\right)=\left(e,s',s^{\prime-1}\right)$ are
both relators). We obtain a Desargues configuration. Denoting $x=s^{\prime\prime}$,
we have 
\[
f\left(s_{k,m}^{\prime-1},s_{i,m},x_{k,i}\right)=f\left(s_{i,m},s_{m,k}^{\prime},s_{k,i}^{\prime\prime}\right)=2,
\]
and \cref{lem:PDG_Desargues} implies that also $f\left(s_{i,j},s_{j,k}^{\prime},s_{k,i}^{\prime\prime}\right)=2$. 

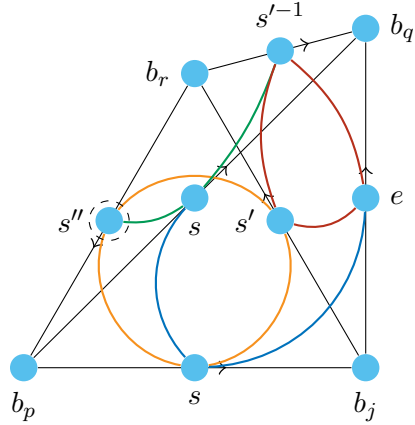
\begin{center}
\begin{tikzpicture}[scale=1.5,
	vert/.style = {
		circle,
		minimum size = 2mm,
		fill=Cerulean!60,
		font=\itshape
	},
    midarrow/.style={
    postaction={decorate},
    decoration={
      markings,
      mark=at position 0.6 with {\arrow[scale=1.5]{>}}
        }
    },
    midarrowShort/.style={
    postaction={decorate},
    decoration={
      markings,
      mark=at position 0.7 with {\arrow[scale=1.5]{>}}
        }
    },
    midarrowLong/.style={
    postaction={decorate},
    decoration={
      markings,
      mark=at position 0.57 with {\arrow[scale=1.5]{>}}
        }
    }
    ]
    \node (v1)  at (-3,0)		    [vert,label=below:$b_p$]{};
    \node (v2)  at (0,0)			[vert,label=below:$b_j$]{};

    \node (v3)  at (-1.5,2.6)		[vert,label=left:$b_r$]{};
    \node (v4)  at (0,3)            [vert,label=right:$b_q$]{};

    \draw[midarrow]         (v1) -- (v2);
    \draw[midarrow]         (v3) -- (v1);
    \draw[midarrow]         (v2) -- (v3);
    \draw[midarrow]         (v2) -- (v4);
    \draw[midarrow]         (v1) -- (v4);
    \draw[midarrowShort]    (v3) -- (v4);
    
    \node (m23) at (-0.75,1.3)      [vert,label=left:$s'$]{};
    \node (m13) at (-2.25,1.3)      [vert,label=left:$s''$]{};
    \node (m12) at (-1.5,0)		    [vert,label=below:$s$]{};
    
    \node (m24) at (0,1.5)          [vert,label=right:$e$]{};
    \node (m14) at (-1.5,1.5)       [vert,label=below:$s$]{};
    \node (m34) at (-0.75,2.8)      [vert,label=above:$s^{\prime -1}$]{};
    
    \draw[BurntOrange,thick]
        (m23) to [bend left=50] (m12) 
        (m12) to [bend left=50] (m13) 
        (m13) to [bend left=50] (m23);
    \draw[RoyalBlue,thick]
        (m12) to [bend left=42] (m14)
        (m24) to [bend left=42] (m12);
    \draw[ForestGreen,thick]
        (m14) to [bend right=10] (m34)
        (m13) to [bend right=20] (m14);
    \draw[BrickRed,thick]
        (m34) to [bend left=20] (m24)
        (m24) to [bend left=30] (m23)
        (m23) to [bend left=20] (m34);
    \draw[dashed] (m13) circle[radius=0.175];
\end{tikzpicture}
\end{center}

Let us say that an index triple $\left(i,j,k\right)$ is \emph{good}
if $f\left(s_{i,j},s_{j,k}^{\prime},s_{k,i}^{\prime\prime}\right)=2$.
Our argument shows that if a triple $\left(i,j,k\right)$ is good
and $m\in\{1,\ldots,4\}$ is the index not appearing in $\left(i,j,k\right)$
then also $\left(i,m,k\right)$ is good (in terms of the notation
above, we started by assuming the triple $\left(i,m,k\right)$ is
good and concluded that $\left(i,j,k\right)$ is; the statement here
is identical up to relabeling). Note that a cyclic shift of a good
triple is also good. We'll show that if some triple $\left(i,j,k\right)$
is good then they all are.

Starting from the three cyclic shifts of $\left(i,j,k\right)$ and
applying this result to each, we find that also $\left(i,m,k\right)$,
$\left(k,m,j\right)$, and $\left(j,m,i\right)$ are good. Applying
it again to the cyclic shifts of $\left(k,m,j\right)$, we find that
$\left(j,i,m\right)$ and $\left(m,i,k\right)$ are good. Finally,
apply the result to the cyclic shift $\left(k,m,i\right)$ of $\left(m,i,k\right)$
to obtain $\left(k,j,i\right)$, and also to the cyclic shift $\left(j,i,k\right)$
of $\left(i,k,j\right)$ to obtain $\left(j,m,k\right)$.

Taking the lexicographically-first cyclic shift of each triple we
found (in the same order in which we found them), we have:
\[
\left(i,j,k\right),\quad\left(i,m,k\right),\quad\left(j,k,m\right),\quad\left(i,j,m\right),\quad\left(i,m,j\right),\quad\left(i,k,m\right),\quad\left(i,k,j\right),\quad\left(j,m,k\right),
\]
and these are representatives for all eight triples of distinct indices
in $\left\{ i,j,k,m\right\} $ modulo cyclic shift.
\end{proof}
\begin{thm}
[Uniqueness up to parallelism]\label[theorem]{thm:PDG_prod_uniqueness}
Let $\left(E,f,S\right)$ be an almost entropic PDG of rank $4$ and
let $s,s'\in S$. If there exist $t,t'\in S$ which are both geometric
products of $s,s'$ then $f\left(t_{i,j},t_{i,j}^{\prime}\right)=1$
for all distinct $1\le i,j\le4$.

In the other direction, if $t\in S$ is a geometric product of $s,s'$
and f$\left(t_{i,j},t_{i,j}^{\prime}\right)=1$ for some $1\le i,j\le4$
then $t^{\prime}$ is also a geometric product of $s,s'$. In particular,
$f\left(t_{k,l},t_{k,l}^{\prime}\right)=1$ for all $1\le k,l\le4$. 
\end{thm}

\begin{proof}
Let $1\le k\le4$ be an index distinct from $i$ and $j$.

Observe that $f\left(t_{i,j},t_{i,j}^{\prime},u_{j,k}\right)=f\left(t_{i,j},t_{i,j}^{\prime}\right)+f\left(u_{j,k}\right)$
for any $u\in S$: we have 
\[
\mathrm{cl}\left(b_{i},b_{j}\right)\cup\left\{ u_{j,k}\right\} \subset\mathrm{cl}\left(b_{i},b_{j},u_{j,k}\right)=\mathrm{cl}\left(b_{i},b_{j},b_{k}\right),
\]
(because $b_{k}\in\mathrm{cl}\left(b_{j},u_{j,k}\right)$ and also
$u_{j,k}\in\mathrm{cl}\left(b_{j},b_{k}\right)$,) and hence 
\[
f\left(\mathrm{cl}\left(b_{i},b_{j}\right)\cup\left\{ u_{j,k}\right\} \right)-f\left(\mathrm{cl}\left(b_{i},b_{j}\right)\right)=f\left(b_{i},b_{j},b_{k}\right)-f\left(b_{i},b_{j}\right)=1.
\]
Since $t_{i,j},t_{i,j}^{\prime}\in\mathrm{cl}\left(b_{i},b_{j}\right)$,
it follows that $f\left(t_{i,j},t_{i,j}^{\prime},u_{j,k}\right)-f\left(t_{i,j},t_{i,j}^{\prime}\right)\ge1=f\left(u_{j,k}\right)$
by the diminishing returns property, but also 
\[
f\left(t_{i,j},t_{i,j}^{\prime},u_{j,k}\right)-f\left(t_{i,j},t_{i,j}^{\prime}\right)\le f\left(u_{j,k}\right)-f\left(\emptyset\right)=1.
\]

Now observe that since $\left(s,s',t^{-1}\right)$ and $\left(s,s',t^{\prime-1}\right)$
are relators of the PDG, so are $\left(t,s^{\prime-1},s^{-1}\right)$
and $\left(t^{\prime},s^{\prime-1},s^{-1}\right)$. Specializing to
$u_{j,k}=s_{j,k}^{\prime-1}$, we find that 
\[
\mathrm{cl}\left(t_{i,j},s_{j,k}^{\prime-1}\right)\supseteq\left\{ t_{i,j},s_{j,k}^{\prime-1},s_{k,i}^{-1}\right\} 
\]
and since $t_{i,j}^{\prime}\in\mathrm{cl}\left(s_{j,k}^{\prime-1},s_{k,i}^{-1}\right)$
we also have $t_{i,j}^{\prime}\in\mathrm{cl}\left(t_{i,j}s_{j,k}^{\prime-1}\right)$.
Hence
\[
2=f\left(t_{i,j},s_{j,k}^{\prime-1}\right)=f\left(t_{i,j},t_{i,j}^{\prime},s_{j,k}^{\prime-1}\right)=f\left(t_{i,j},t_{i,j}^{\prime}\right)+f\left(s_{j,k}^{\prime-1}\right)=f\left(t_{i,j},t_{i,j}^{\prime}\right)+1
\]
where $2=f\left(t_{i,j},s_{j,k}^{\prime-1}\right)$ by \cref{cor:pair_independence},
and $f\left(s_{j,k}^{\prime-1}\right)=1$ by (\ref{PDG:nondegeneracy}).
Therefore $f\left(t_{i,j},t_{i,j}^{\prime}\right)=1$.

In the other direction, suppose $t$ is a geometric product of $s,s'$
and $f\left(t_{i,j},t_{i,j}^{\prime}\right)=1$. Then $f\left(t_{i,j}^{\prime},s_{j,k}^{\prime-1},s_{k,i}^{-1}\right)=f\left(t_{i,j},s_{j,k}^{\prime-1},s_{k,i}^{-1}\right)=2$,
where the first equality holds because the closures of the two sets
are equal. By \cref{lem:finding_relators}, $\left(t^{\prime},s^{\prime-1},s^{-1}\right)$
is a relator, and hence so is $\left(s,s^{\prime},t^{\prime-1}\right)$.
\end{proof}
\begin{lem}
Let $\left(E,f,S\right)$ be an almost entropic PDG of rank $4$.
Let $s,s'\in S$ (not necessarily distinct). Then there exists an
almost entropic rank $4$ PDG $\left(\tilde{E},\tilde{f},\tilde{S}\right)$
which is an extension of $\left(E,f,S\right)$ and an element $t\in\tilde{S}$
such that $\left(s,s',t\right)$ is a relator of the extension.
\end{lem}

\begin{proof}
If there is already a $t\in S$ such that $\left(s,s',t\right)$ is
a relator, we are done (the wanted PDG is $\left(E,f,S\right)$).
Otherwise we do this in two steps: first we construct an extension
and then prove that it is the required PDG.

The extension is constructed as follows: we use the relations $\left(s^{-1},s,e\right)$
and $\left(e,s^{\prime},s^{\prime-1}\right)$ and apply Desargues'
theorem several times to enlarge the polymtaroid. For each ordered
triple $\left(i,j,k\right)$ of distinct indices in $\{1,\ldots,4\}$,
let $m$ denote the missing index. \Cref{lem:PDG_Desargues} implies
that $\left(b_{k},b_{i},b_{j}\right)$ and $\left(s_{k,m}^{\prime-1},s_{i,m},e_{j,m}\right)$
are in perspective from $b_{m}$, and together with the intersection
points $s_{i,j}\in\mathrm{cl}\left(s_{i,m},e_{j,m}\right)$ and $s_{j,k}^{\prime}\in\mathrm{cl}\left(s_{k,m}^{\prime-1},e_{j,m}\right)$
they form a Desargues configuration:

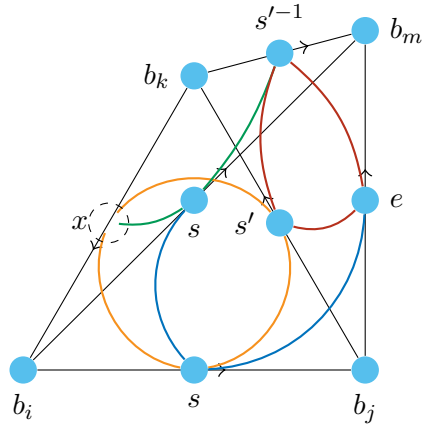
\begin{center}
\begin{tikzpicture}[scale=1.5,
	vert/.style = {
		circle,
		minimum size = 2mm,
		fill=Cerulean!60,
		font=\itshape
	},
    midarrow/.style={
    postaction={decorate},
    decoration={
      markings,
      mark=at position 0.6 with {\arrow[scale=1.5]{>}}
        }
    },
    midarrowShort/.style={
    postaction={decorate},
    decoration={
      markings,
      mark=at position 0.7 with {\arrow[scale=1.5]{>}}
        }
    },
    midarrowLong/.style={
    postaction={decorate},
    decoration={
      markings,
      mark=at position 0.57 with {\arrow[scale=1.5]{>}}
        }
    }
    ]
    \node (v1)  at (-3,0)		    [vert,label=below:$b_i$]{};
    \node (v2)  at (0,0)			[vert,label=below:$b_j$]{};

    \node (v3)  at (-1.5,2.6)		[vert,label=left:$b_k$]{};
    \node (v4)  at (0,3)            [vert,label=right:$b_m$]{};

    \draw[midarrow]         (v1) -- (v2);
    \draw[midarrow]         (v3) -- (v1);
    \draw[midarrow]         (v2) -- (v3);
    \draw[midarrow]         (v2) -- (v4);
    \draw[midarrow]         (v1) -- (v4);
    \draw[midarrowShort]    (v3) -- (v4);

    \node (m23) at (-0.75,1.3)      [vert,label=left:$s'$]{};
    \node (m13) at (-2.25,1.3)      [label=left:$x$]{};
    \node (m12) at (-1.5,0)		    [vert,label=below:$s$]{};
    
    \node (m24) at (0,1.5)          [vert,label=right:$e$]{};
    \node (m14) at (-1.5,1.5)       [vert,label=below:$s$]{};
    \node (m34) at (-0.75,2.8)      [vert,label=above:$s^{\prime -1}$]{};
	    
    \draw[BurntOrange,thick]
        (m23) to [bend left=50] (m12) 
        (m12) to [bend left=50] (m13) 
        (m13) to [bend left=50] (m23);
    \draw[RoyalBlue,thick]
        (m12) to [bend left=42] (m14)
        (m24) to [bend left=42] (m12);
    \draw[ForestGreen,thick]
        (m14) to [bend right=10] (m34)
        (m13) to [bend right=20] (m14);
    \draw[BrickRed,thick]
        (m34) to [bend left=20] (m24)
        (m24) to [bend left=30] (m23)
        (m23) to [bend left=20] (m34);
    \draw[dashed] (m13) circle[radius=0.175];
\end{tikzpicture}
\end{center}

Desargues' theorem (\cref{thm:aent_Desargues}) guarantees that there
exists an extension $\left(\tilde{E},\tilde{f}\right)$ of $\left(E,f\right)$
with a new point $x_{k,i}\in\tilde{E}$ that satisfies: 
\[
\begin{cases}
\tilde{f}(x_{k,i})=1,\\
\tilde{f}\left(b_{k},b_{i},x_{k,i}\right)=2=\tilde{f}\left(s_{k,m}^{\prime-1},s_{i,m},x_{k,i}\right),\\
\tilde{f}\left(b_{k},x_{k,i}\right)=2=\tilde{f}\left(s_{k,m}^{\prime-1},x_{k,i}\right),\\
\tilde{f}\left(s_{i,j},s_{j,k}^{\prime},x_{k,i}\right)=2.
\end{cases}
\]

We iteratively extend $\left(E,f\right)$ in this way: for each ordered
index pair $\left(k,i\right)$ in turn, we choose $j$ arbitrarily
among the missing indices $\left\{ 1,\ldots,4\right\} \setminus\left\{ i,k\right\} $
and construct the extension as above. At the end of this process we
have an extension $\left(\tilde{E},\tilde{f}\right)$ with $12$ additional
elements (we have added $x_{i,j}$ and $x_{j,i}$ in the closure of
each pair $\left\{ b_{i},b_{j}\right\} $). 

Define $\tilde{S}=S\sqcup\left\{ t,t^{-1}\right\} $, set $t_{i,j}=x_{i,j}$
and similarly $t_{i,j}^{-1}=x_{j,i}$ whenever $1\le i<j\le r$. Extend
the involution from $S$ to $\tilde{S}$ in the obvious way, (so $\left(t\right)^{-1}=t^{-1}$
and $\left(t_{i,j}\right)^{-1}=\left(t^{-1}\right)_{i,j}$,) and define
$t_{j,i}=t_{i,j}^{-1}$ as well as $\left(t^{-1}\right)_{j,i}=t_{i,j}$
whenever $i<j$. Then $\left(\tilde{E},\tilde{f},\tilde{S}\right)$
is a PDG: conditions (\ref{PDG:S},\ref{PDG:B},\ref{PDG:E}) are clear
from the corresponding conditions for $\left(E,f,S\right)$. Condition
(\ref{PDG:nondegeneracy}) is the statement that $\tilde{f}\left(t_{i,j}\right)=\tilde{f}\left(x_{i,j}\right)=1$
and $\tilde{f}\left(b_{i},t_{i,j}\right)=\tilde{f}\left(b_{i},x_{i,j}\right)=2$
for all $i,j$: these conditions are part of the defining conditions
of $x_{i,j}$ (from our application of Desargues' theorem). Condition
(\ref{PDG:frame}) is that $t_{i,j}\in\mathrm{cl}\left(b_{i},b_{j}\right)$
for all $i,j$, and this is equivalent to $\tilde{f}\left(b_{i},b_{j},x_{i,j}\right)=2$,
another of the defining conditions of $t_{i,j}=x_{i,j}$. 

This proves $\left(\tilde{E},\tilde{f},\tilde{S}\right)$ is an incoherent
PDG. We are left to prove coherence, i.e. condition (\ref{PDG:coherence}),
which is the statement that
\[
\tilde{f}\left(t_{i,j},t_{j,k}^{-1},e_{k,i}\right)=2
\]
for all $1\le i,j,k\le4$. Observe that $\left(s,s',t\right)$ is
a relator of $\left(\tilde{E},\tilde{f},\tilde{S}\right)$. Hence
$\left(s',t,s\right)$ and $\left(s^{-1},t^{-1},s^{\prime-1}\right)$
are also relators. Let $\left(i,j,k\right)$ be a triple of distinct
indices, and let $1\le m\le4$ be the missing index. By \cref{lem:PDG_Desargues},
the triples $\left(b_{k},b_{i},b_{j}\right)$ and $\left(s_{m,k}^{\prime},s_{m,i}^{\prime},s_{j,m}\right)$
are in perspective from $b_{m}$, and they form a Desargues configuration
with the intersection points $t_{i,j}\in\mathrm{cl}\left(b_{i},b_{j}\right)\cap\mathrm{cl}\left(s_{m,i}^{\prime},s_{j,m}\right)$
and $t_{j,k}^{-1}\in\mathrm{cl}\left(b_{j},b_{k}\right)\cap\mathrm{cl}\left(s_{j,m},s_{m,k}^{\prime}\right)$:

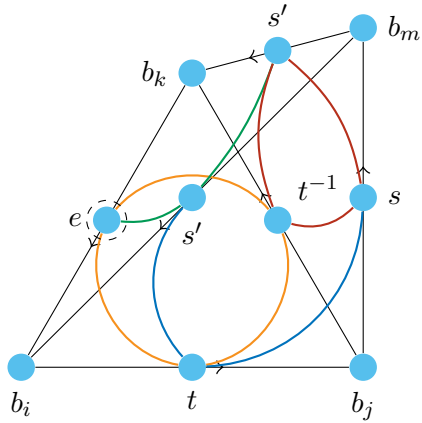
\begin{center}
\begin{tikzpicture}[scale=1.5,
	vert/.style = {
		circle,
		minimum size = 2mm,
		fill=Cerulean!60,
		font=\itshape
	},
    midarrow/.style={
    postaction={decorate},
    decoration={
      markings,
      mark=at position 0.6 with {\arrow[scale=1.5]{>}}
        }
    },
    midarrowShort/.style={
    postaction={decorate},
    decoration={
      markings,
      mark=at position 0.7 with {\arrow[scale=1.5]{>}}
        }
    },
    midarrowLong/.style={
    postaction={decorate},
    decoration={
      markings,
      mark=at position 0.57 with {\arrow[scale=1.5]{>}}
        }
    }
    ]
    \node (v1)  at (-3,0)		    [vert,label=below:$b_i$]{};
    \node (v2)  at (0,0)			[vert,label=below:$b_j$]{};

    \node (v3)  at (-1.5,2.6)		[vert,label=left:$b_k$]{};
    \node (v4)  at (0,3)            [vert,label=right:$b_m$]{};

    \draw[midarrow]         (v1) -- (v2);
    \draw[midarrow]         (v3) -- (v1);
    \draw[midarrow]         (v2) -- (v3);
    \draw[midarrow]         (v2) -- (v4);
    \draw[midarrow]         (v4) -- (v1);
    \draw[midarrowShort]    (v4) -- (v3);
    
    \node (m23) at (-0.75,1.3)      [vert,label=north east:$t^{-1}$]{};
    \node (m13) at (-2.25,1.3)      [vert,label=left:$e$]{};
    \node (m12) at (-1.5,0)		    [vert,label=below:$t$]{};
    
    \node (m24) at (0,1.5)          [vert,label=right:$s$]{};
    \node (m14) at (-1.5,1.5)       [vert,label=below:$s'$]{};
    \node (m34) at (-0.75,2.8)      [vert,label=above:$s'$]{};

    \draw[BurntOrange,thick]
        (m23) to [bend left=50] (m12) 
        (m12) to [bend left=50] (m13) 
        (m13) to [bend left=50] (m23);
    \draw[RoyalBlue,thick]
        (m12) to [bend left=42] (m14)
        (m24) to [bend left=42] (m12);
    \draw[ForestGreen,thick]
        (m14) to [bend right=10] (m34)
        (m13) to [bend right=20] (m14);
    \draw[BrickRed,thick]
        (m34) to [bend left=20] (m24)
        (m24) to [bend left=30] (m23)
        (m23) to [bend left=20] (m34);
    \draw[dashed] (m13) circle[radius=0.175];
\end{tikzpicture}
\end{center}

Further, (by the same lemma,) since $f\left(e_{k,i},s_{m,k}^{\prime},s_{m,i}^{\prime}\right)=2$,
we have $f\left(t_{i,j},t_{j,k}^{-1},e_{k,i}\right)=2$ as required.
\end{proof}
\begin{thm}
[A product-closed extension]\label[theorem]{thm:prod_extension}Let
$\left(E,f,S\right)$ be a finite almost entropic PDG of rank $4$.
There exists an extension $\left(\tilde{E},\tilde{f},\tilde{S}\right)$
of $\left(E,f,S\right)$, which is again an almost entropic PDG of
rank $4$, such that for all $s,s'\in\tilde{S}$ (not necessarily
distinct) there exists a geometric product $t\in\tilde{S}$.
\end{thm}

This also holds in the countable case. The details are routine but
slightly longer, and the claim is not used in this paper, so it is
omitted.
\begin{proof}
We construct a chain of extensions using the previous lemma: denote
$\left(E_{0},f_{0},S_{0}\right)=\left(E,f,S\right)$. Given $\left(E_{n},f_{n},S_{n}\right)$,
enumerate all ordered pairs $\left(s,s'\right)$ of elements in $S$
(including pairs with $s=s'$) as $\left(s_{1},s_{1}^{\prime}\right),\ldots,\left(s_{k},s_{k}^{\prime}\right)$.
Then apply the previous lemma to construct a chain of extensions 
\[
\left(E_{n,0},f_{n,0},S_{n,0}\right)=\left(E_{n},f_{n},S_{n}\right),\quad\ldots,\quad\left(E_{n,k},f_{n,k},S_{n,k}\right)\eqqcolon\left(E_{n+1},f_{n+1},S_{n+1}\right)
\]
of $\left(E_{n},f_{n},S_{n}\right)$, where $\left(E_{n,i+1},f_{n,i+1},S_{n,i+1}\right)$
is an extension of $\left(E_{n,i},f_{n,i},S_{n,i}\right)$ in which
a geometric product of $\left(s_{i},s_{i}^{\prime}\right)$ exists.
Having done this for all $n$, define $\tilde{E}=\bigcup_{n}E_{n}$,
and define
\[
\tilde{f}:P\left(\tilde{E}\right)\rightarrow\mathbb{R}_{\ge0}
\]
by $\tilde{f}\left(A\right)=\lim_{n}f_{n}\left(A\cap E_{n}\right)$
(the limit exists because it is a monotone sequence, and always bounded
above by $4$, which is the value of $f_{n}\left(E_{n}\right)$ for
all $n$). Similarly define $\tilde{S}=\bigcup_{n}S_{n}$, and note
that the involution on $S_{n}$ agrees with the involution on the
subset $S_{k}$ for all $n\ge k$, so also $\tilde{S}$ is naturally
equipped with an involution. The pair $\left(\tilde{E},\tilde{f}\right)$
is indeed a polymatroid; it is of finite type because the limit $\lim_{n}f_{n}\left(A\cap E_{n}\right)$
equals $\sup_{n}\tilde{f}\left(A\cap E_{n}\right)$, and any finite
$F\subset A$ is contained in some $E_{n}$. Since we have $\tilde{f}\left(F\right)\le\tilde{f}\left(A\cap E_{n}\right)$
by monotonicity, we obtain
\[
\tilde{f}\left(A\right)=\sup_{F\subseteq A\text{ finite}}\tilde{f}\left(F\right).
\]
Also the PDG properties (\ref{PDG:S}-\ref{PDG:coherence}) for $\left(\tilde{E},\tilde{f},\tilde{S}\right)$
are immediate from the definition, since they hold for each $n$.
If $\left(s,s',t^{-1}\right)$ is a relator of some $\left(E_{n},f_{n},S_{n}\right)$
then it is also a relator of $\left(\tilde{E},\tilde{f},\tilde{S}\right)$.
Since any $s,s'\in\tilde{S}$ are elements of some $S_{n}$, and have
a geometric product in $\left(E_{n+1},f_{n+1},S_{n+1}\right)$, they
have a geometric product in $\left(\tilde{E},\tilde{f},\tilde{S}\right)$
as required.
\end{proof}

\subsection{Recovering a group}

\label[section]{sec:group_from_geometric_product}

The goal of this section is to recover a group from an almost entropic
PDG of rank $4$.
\begin{defn}
If $\left(E,f,S\right)$ is a rank-4 almost entropic PDG and $t,t'\in S$,
write $t\sim t'$ if $t'_{i,j}\in\mathrm{cl}\left(t_{i,j}\right)$
for some $1\le i,j\le4$ (and hence all $i,j$, by \cref{thm:PDG_prod_uniqueness}).
This is the \emph{parallelism relation}, and it is straightforward
that it is an equivalence relation. Denote the equivalence class (or
\emph{parallelism class}) of $t\in S$ by $\left[t\right]$.

If $t\in S$ is a geometric product of $s,s'\in S$, denote $t\in s\cdot s'$
(so that $s\cdot s'$ is the set of geometric products of $s,s'$).
By \cref{thm:PDG_prod_uniqueness}, the geometric product descends
to a well-defined function on the parallelism classes, and we denote
$\left[t\right]=\left[s\right]\cdot\left[s'\right]$.
\end{defn}

\begin{thm}
\label[theorem]{thm:group_from_PDG}Let $\left(E,f,S\right)$ be an
almost entropic PDG of rank $4$ which is closed under geometric products.
Then $\left(S/\mathord{}{\sim},\cdot\right)$ is a group, and the
inverse of each $\left[s\right]$ (where $s\in S$) is $\left[s^{-1}\right]$.
\end{thm}

\begin{proof}
The geometric product $\cdot$ is a function $\left(S/\mathord{}{\sim}\right)\times\left(S/\mathord{}{\sim}\right)\rightarrow S/\mathord{}{\sim}$:
at least one product is defined for any pair $\left(\left[s\right],\left[s'\right]\right)$
by the assumption that $\left(E,f,S\right)$ is closed under geometric
products, and it is single-valued by \cref{thm:PDG_prod_uniqueness}
(any two products of $s,s'\in S$ are parallel). Inverses exist by
(\ref{PDG:coherence}): given $s\in S$, since $\left(s,s^{-1},e\right)$
is a relator we have that 
\[
\left[e^{-1}\right]=\left[e\right]=\left[s\right]\cdot\left[s^{-1}\right].
\]
Similarly (since $\left(\right)^{-1}$ is an involution) also 
\[
\left[e^{-1}\right]=\left[e\right]=\left[s^{-1}\right]\cdot\left[\left(s^{-1}\right)^{-1}\right]=\left[s^{-1}\right]\cdot\left[s\right].
\]
Also, again by (\ref{PDG:coherence}), $\left[e\right]$ is a neutral
element for $\left(\cdot\right)$: since $\left(s,s^{-1},e\right)$
is a relator, also $\left(e,s,s^{-1}\right)$ is a relator (by taking
a cyclic shift); this yields $\left[e\right]\cdot\left[s\right]=\left[s\right]$.
Since $\left(s^{-1},s,e\right)$ is a relator (this time by (\ref{PDG:coherence})
applied to $s^{-1}$), the cyclic shift $\left(s,e,s^{-1}\right)$
is also a relator, which yields $\left[s\right]\cdot\left[e\right]=\left[s\right]$.

We now prove associativity. Let $x,y,z\in S$. Let $s\in S$ satisfy
$\left[s\right]=\left[x\right]\cdot\left[y\right]$ and let $v\in S$
satisfy $\left[v\right]=\left[y\right]\cdot\left[z\right]$. Then
$\left(x,y,s^{-1}\right)$ and $\left(y,z,v^{-1}\right)$ are relators,
so that (by taking cyclic shifts) $\left(s^{-1},x,y\right)$ and $\left(y^{-1},v,z^{-1}\right)$
are also relators of $\left(E,f,S\right)$. By \cref{lem:PDG_Desargues},
$\left(z_{3,4}^{-1},s_{1,4},y_{2,4}\right)$ and $\left(b_{3},b_{1},b_{2}\right)$
are in perspective from $b_{4}$, and they form a Desargues configuration
with the intersection points $x_{1,2}\in\mathrm{cl}\left(s_{1,4},y_{2,4}\right)\cap\mathrm{cl}\left(b_{1},b_{2}\right)$
and $v_{2,3}\in\mathrm{cl}\left(y_{2,4},z_{3,4}^{-1}\right)\cap\mathrm{cl}\left(b_{2},b_{3}\right)$:

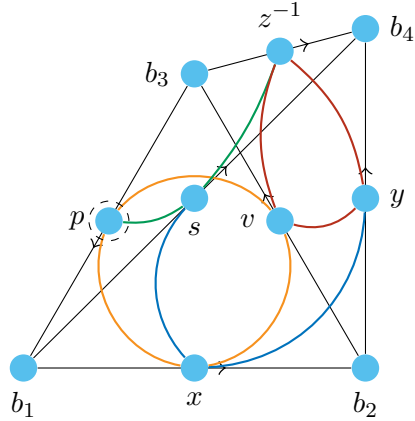
\begin{center}
\begin{tikzpicture}[scale=1.5,
	vert/.style = {
		circle,
		minimum size = 2mm,
		fill=Cerulean!60,
		font=\itshape
	},
    midarrow/.style={
    postaction={decorate},
    decoration={
      markings,
      mark=at position 0.6 with {\arrow[scale=1.5]{>}}
        }
    },
    midarrowShort/.style={
    postaction={decorate},
    decoration={
      markings,
      mark=at position 0.7 with {\arrow[scale=1.5]{>}}
        }
    },
    midarrowLong/.style={
    postaction={decorate},
    decoration={
      markings,
      mark=at position 0.57 with {\arrow[scale=1.5]{>}}
        }
    }
    ]
    \node (v1)  at (-3,0)		    [vert,label=below:$b_1$]{};
    \node (v2)  at (0,0)			[vert,label=below:$b_2$]{};

    \node (v3)  at (-1.5,2.6)		[vert,label=left:$b_3$]{};
    \node (v4)  at (0,3)            [vert,label=right:$b_4$]{};

    \draw[midarrow]         (v1) -- (v2);
    \draw[midarrow]         (v3) -- (v1);
    \draw[midarrow]         (v2) -- (v3);
    \draw[midarrow]         (v2) -- (v4);
    \draw[midarrow]         (v1) -- (v4);
    \draw[midarrowShort]    (v3) -- (v4);

    \node (m23) at (-0.75,1.3)      [vert,label=left:$v$]{};
    \node (m13) at (-2.25,1.3)      [vert,label=left:$p$]{};
    \node (m12) at (-1.5,0)		    [vert,label=below:$x$]{};
    
    \node (m24) at (0,1.5)          [vert,label=right:$y$]{};
    \node (m14) at (-1.5,1.5)       [vert,label=below:$s$]{};
    \node (m34) at (-0.75,2.8)      [vert,label=above:$z^{-1}$]{};

    \draw[BurntOrange,thick]
        (m23) to [bend left=50] (m12) 
        (m12) to [bend left=50] (m13) 
        (m13) to [bend left=50] (m23);
    \draw[RoyalBlue,thick]
        (m12) to [bend left=42] (m14)
        (m24) to [bend left=42] (m12);
    \draw[ForestGreen,thick]
        (m14) to [bend right=10] (m34)
        (m13) to [bend right=20] (m14);
    \draw[BrickRed,thick]
        (m34) to [bend left=20] (m24)
        (m24) to [bend left=30] (m23)
        (m23) to [bend left=20] (m34);
    \draw[dashed] (m13) circle[radius=0.175];
\end{tikzpicture}
\end{center}

The last part of \cref{lem:PDG_Desargues} states that if $p_{3,1}\in\mathrm{cl}\left(z_{3,4}^{-1},s_{1,4}\right)=\mathrm{cl}\left(s_{1,4},z_{4,3}\right)$,
we have $f\left(x_{1,2},v_{2,3},p_{3,1}\right)=2$. Let $p\in S$
satisfy 
\[
\left[p^{-1}\right]=\left[s\right]\cdot\left[z\right]=\left(\left[x\right]\cdot\left[y\right]\right)\cdot\left[z\right],
\]
so that $\left(s,z,\left(p^{-1}\right)^{-1}\right)=\left(s,z,p\right)$
is a relator. This is equivalent to $f\left(s_{1,4},z_{4,3},p_{3,1}\right)=2$,
which implies $p_{3,1}\in\mathrm{cl\left(z_{3,4}^{-1},s_{1,4}\right)}$,
and we find $f\left(x_{1,2},v_{2,3},p_{3,1}\right)=2$, i.e. $\left(x,v,p\right)$
is a relator (by \cref{lem:finding_relators}). But this means
\[
\left[p^{-1}\right]=\left[x\right]\cdot\left[v\right]=\left[x\right]\cdot\left(\left[y\right]\cdot\left[z\right]\right),
\]
and we have obtained $\left(\left[x\right]\cdot\left[y\right]\right)\cdot\left[z\right]=\left[x\right]\cdot\left(\left[y\right]\cdot\left[z\right]\right)$
as required (both expressions equal $\left[p^{-1}\right]$).
\end{proof}

\section{The lifting theorem}

\label[section]{sec:PDG_lifting}

Basic almost entropic geometry can be used to ``lift'' almost entropic
PDGs of rank $3$ to almost entropic PDGs of rank $4$. In principle
one could simply work with rank $4$ PDGs from the outset, but in
the context of \cite{KY22} it was simpler to work with rank $3$
geometries. We prove the lifting theorem here in order to use the
constructions of that paper.
\begin{thm}
\label[theorem]{thm:lifting} Let $\left(E,f,S\right)$ be an almost
entropic PDG of rank $3$. Then there exists an almost entropic PDG
$\left(\tilde{E},\tilde{f},S\right)$ of rank $4$ (on the same generating
set) that extends $\left(E,f\right)$.
\end{thm}

\begin{proof}
We do this in several steps. Let $\left(E',f'\right)$ be given by
applying the copy lemma to $\left(E,f\right)$ where the base is $A\coloneqq\left\{ b_{1},b_{2}\right\} \cup\left\{ s_{1,2}\mid s\in S\right\} $
and its complement $C=E\setminus A$ is the set to be copied. Identify
$E$ with its image in $E'$, and for each $x\in C$ denote its copy
in $E'$ by $x'$. We set $b_{4}\coloneqq b_{3}^{\prime}$, $s_{1,4}\coloneqq s_{1,3}^{\prime}$
for all $s\in S$, and similarly $s_{2,4}\coloneqq s_{2,3}^{\prime}$
for all $s\in S$. (These are all the elements of $E'$: for the time
being we have no elements named ``$s_{3,4}$''.) The involution
extends naturally to the new elements, so we have $s_{4,1}\coloneqq s_{1,4}^{-1}=\left(s_{1,3}^{-1}\right)^{\prime}=s_{3,1}^{\prime}$
for all $s\in S$ and similarly $s_{4,2}\coloneqq s_{3,2}^{\prime}$.

Before extending the polymatroid further we verify the conditions
on $f'$ for those elements that are present: we need to check conditions
(\ref{PDG:B}) (independence of $B$), (\ref{PDG:frame}) (every element
is in the closure of some $b_{i},b_{j}$), (\ref{PDG:nondegeneracy})
(nondegeneracy conditions), (\ref{PDG:coherence}) (coherent relators)
in the definition of PDGs. Condition (\ref{PDG:B}) holds: independence
of $C,C'$ over $A$ implies that 
\[
f\left(b_{1},b_{2},b_{3},b_{4}\right)-f\left(b_{1},b_{2}\right)=\left[f\left(b_{1},b_{2},b_{3}\right)-f\left(b_{1},b_{2}\right)\right]+\left[f\left(b_{1},b_{2},b_{4}\right)-f\left(b_{1},b_{2}\right)\right]
\]
or $f\left(b_{1},b_{2},b_{3},b_{4}\right)-2=2$. Also $f\left(b_{4}\right)=f\left(b_{3}\right)=1$.
So $\left\{ b_{1},\ldots,b_{4}\right\} $ are independent.

Condition (\ref{PDG:frame}) holds for $\{i,j\}=\{1,4\}$ and $\{i,j\}=\{2,4\}$
by the copy lemma: $f'\restriction_{A\cup C}\simeq f'\restriction_{A\cup C'}$
via the ``obvious'' bijection $A\cup C\rightarrow A\cup C'$. Hence
we have, for $\{i,j\}=\{1,4\}$,
\[
f'\left(\left\{ b_{1},b_{4}\right\} \cup\left\{ s_{1,4}\mid s\in S\right\} \right)=f'\left(\left\{ b_{1},b_{3}\right\} \cup\left\{ s_{1,3}\mid s\in S\right\} \right)
\]
\[
=f\left(\left\{ b_{1},b_{3}\right\} \cup\left\{ s_{1,3}\mid s\in S\right\} \right)=2
\]
(and the case $\{i,j\}=\{2,4\}$ is similar). Condition (\ref{PDG:coherence})
holds for $\{i,j,k\}=\{1,2,4\}$ for precisely the same reason. Also
(\ref{PDG:nondegeneracy}) holds for $\{i,j\}=\{1,4\}$ and $\{i,j\}=\{2,4\}$.
It is clear that all of these conditions hold for indices contained
in $\{1,2,3\}$, since $\left(E,f\right)$ is a PDG.

We now extend the polymatroid further to add elements $\left\{ s_{3,4}\mid s\in S\right\} $.
For each $s\in S$ we use Desargues' theorem: denote $\left(i,j,k\right)=\left(3,1,4\right)$
and $m=2$. Then 
\[
f'\left(s_{m,i}^{-1},e_{i,j},s_{j,m}\right)=2=f'\left(s_{m,j}^{-1},s_{j,k},e_{k,m}\right)
\]
(and all these points are in $E'$, since the index pair $\{3,4\}=\{i,k\}$
doesn't appear). \Cref{lem:PDG_Desargues} shows that $\left(b_{k},b_{i},b_{j}\right)$
and $\left(e_{k,m},s_{i,m},s_{j,m}\right)$ are in perspective from
$b_{m}$, and they form a Desargues configuration with the intersection
points $e_{i,j}\in\mathrm{cl}\left(s_{i,m},s_{j,m}\right)\cap\mathrm{cl}\left(b_{i},b_{j}\right)$
and $s_{j,k}\in\mathrm{cl}\left(s_{j,m},e_{k,m}\right)\cap\mathrm{cl}\left(b_{j},b_{k}\right)$:

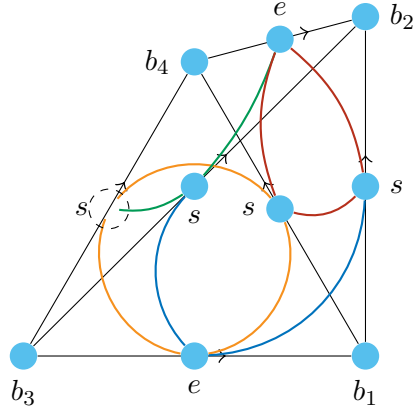
\begin{center}
\begin{tikzpicture}[scale=1.5,
	vert/.style = {
		circle,
		minimum size = 2mm,
		fill=Cerulean!60,
		font=\itshape
	},
    midarrow/.style={
    postaction={decorate},
    decoration={
      markings,
      mark=at position 0.6 with {\arrow[scale=1.5]{>}}
        }
    },
    midarrowShort/.style={
    postaction={decorate},
    decoration={
      markings,
      mark=at position 0.7 with {\arrow[scale=1.5]{>}}
        }
    },
    midarrowLong/.style={
    postaction={decorate},
    decoration={
      markings,
      mark=at position 0.57 with {\arrow[scale=1.5]{>}}
        }
    }
    ]
    \node (v1)  at (-3,0)		    [vert,label=below:$b_3$]{};
    \node (v2)  at (0,0)			[vert,label=below:$b_1$]{};

    \node (v3)  at (-1.5,2.6)		[vert,label=left:$b_4$]{};
    \node (v4)  at (0,3)            [vert,label=right:$b_2$]{};

    \draw[midarrow]         (v1) -- (v2);
    \draw[midarrow]         (v1) -- (v3);
    \draw[midarrow]         (v2) -- (v3);
    \draw[midarrow]         (v2) -- (v4);
    \draw[midarrow]         (v1) -- (v4);
    \draw[midarrowShort]    (v3) -- (v4);
    
    \node (m23) at (-0.75,1.3)      [vert,label=left:$s$]{};
    \node (m13) at (-2.25,1.3)      [label=left:$s$]{};
    \node (m12) at (-1.5,0)		    [vert,label=below:$e$]{};
    
    \node (m24) at (0,1.5)          [vert,label=right:$s$]{};
    \node (m14) at (-1.5,1.5)       [vert,label=below:$s$]{};
    \node (m34) at (-0.75,2.8)      [vert,label=above:$e$]{};
	    
    \draw[BurntOrange,thick]
        (m23) to [bend left=50] (m12) 
        (m12) to [bend left=50] (m13) 
        (m13) to [bend left=50] (m23);
    \draw[RoyalBlue,thick]
        (m12) to [bend left=42] (m14)
        (m24) to [bend left=42] (m12);
    \draw[ForestGreen,thick]
        (m14) to [bend right=10] (m34)
        (m13) to [bend right=20] (m14);
    \draw[BrickRed,thick]
        (m34) to [bend left=20] (m24)
        (m24) to [bend left=30] (m23)
        (m23) to [bend left=20] (m34);
    \draw[dashed] (m13) circle[radius=0.175];
\end{tikzpicture}
\end{center}

By applying \cref{thm:aent_Desargues} to each $s\in S$ and taking
a sequence of extensions, we obtain an extension $\left(\tilde{E},\tilde{f}\right)$
of $\left(E',f'\right)$ such that for each $s\in S$ there exists
$s_{3,4}\in E$ satisfying:
\[
\begin{cases}
\tilde{f}\left(s_{3,4}\right)=1,\\
\tilde{f}\left(s_{3,2},e_{4,2},s_{3,4}\right)=2=\tilde{f}\left(b_{3},b_{4},s_{3,4}\right),\\
\tilde{f}\left(s_{3,2},s_{3,4}\right)=2=\tilde{f}\left(b_{3},s_{3,4}\right),\\
\tilde{f}\left(e_{3,1},s_{1,4},s_{3,4}\right)=2.
\end{cases}
\]
The involution on $S$ extends to the new elements: we set $s_{4,3}=s_{3,4}^{-1}$
for each $s\in S$. This still satisfies (\ref{PDG:S},\ref{PDG:B}),
but now also satisfies (\ref{PDG:E}) and (\ref{PDG:frame}) for all
$s_{i,j}$ ($1\le i,j\le4$). Also (\ref{PDG:nondegeneracy}) holds:
by our application of \cref{thm:aent_Desargues} we have $\tilde{f}\left(s_{3,4}\right)=1$
and $\tilde{f}\left(b_{3},s_{3,4}\right)=2$, and so we need to establish
$\tilde{f}\left(b_{4},s_{3,4}\right)$. We have $\tilde{f}\left(b_{4},s_{3,4}\right)\le\tilde{f}\left(b_{3},b_{4},s_{3,4}\right)=2$,
and on the other hand by diminishing returns
\[
\tilde{f}\left(b_{4},s_{3,4}\right)-\underbrace{\tilde{f}\left(s_{3,4}\right)}_{1}\ge\tilde{f}\left(b_{4},s_{3,2},e_{4,2},s_{3,4}\right)-\underbrace{\tilde{f}\left(s_{3,2},e_{4,2},s_{3,4}\right)}_{2}.
\]
Observe that $b_{2}\in\mathrm{cl}\left(b_{4},e_{4,2}\right)$ and
$b_{3}\in\mathrm{cl}\left(b_{2},s_{3,2}\right)$, so $b_{4},b_{2},b_{3}\in\mathrm{cl}\left(b_{4},s_{3,2},e_{4,2},s_{3,4}\right)$,
and hence $\tilde{f}\left(b_{4},s_{3,2},e_{4,2},s_{3,4}\right)\ge3$.
We have obtained
\[
\tilde{f}\left(b_{4},s_{3,4}\right)-1\ge1,
\]
so $2\ge\tilde{f}\left(b_{4},s_{3,4}\right)\ge2$ as required.

All that is left is (\ref{PDG:coherence}) for the index triples involving
both of $\left\{ 3,4\right\} $. In some cases this follows directly
from our application of \cref{thm:aent_Desargues}: we have $\tilde{f}\left(s_{3,2},e_{4,2},s_{3,4}\right)=2=\tilde{f}\left(e_{3,1},s_{1,4},s_{3,4}\right)$.
Rewriting the expression on the right as $\tilde{f}\left(s_{3,4},s_{4,1}^{-1},e_{1,3}\right)$
and the expression on the right as $\tilde{f}\left(s_{2,3}^{-1},s_{3,4},e_{4,2}\right)$,
we obtain (\ref{PDG:coherence}) for $s$ with $\left(i,j,k\right)=\left(3,4,1\right)$,
and for $s^{-1}$ with $\left(i,j,k\right)=\left(2,3,4\right)$. Observe
that using the involution, (\ref{PDG:coherence}) for $s^{-1}$ with
the index triple $\left(i,j,k\right)$ automatically yields it also
for $s$ with the indices $\left(k,j,i\right)$; also, since we are
quantifying over all $S$, any case of (\ref{PDG:coherence}) that
holds for $s^{-1}$ also holds for $s$. Using this and the cases
of (\ref{PDG:coherence}) we already have, it suffices to prove it
for the four index triples $\left(1,3,4\right)$, $\left(4,1,3\right)$,
$\left(4,2,3\right)$, and $\left(3,4,2\right)$. These four conditions
each follow from an application of \cref{lem:PDG_Desargues}, and
in this instance we simply depict them:

\phantom{}

\hspace{-0.5cm}%
\begin{tabular}{cl}
\begin{tikzpicture}[scale=1.5,
	vert/.style = {
		circle,
		minimum size = 2mm,
		fill=Cerulean!60,
		font=\itshape
	},
    midarrow/.style={
    postaction={decorate},
    decoration={
      markings,
      mark=at position 0.6 with {\arrow[scale=1.5]{>}}
        }
    },
    midarrowShort/.style={
    postaction={decorate},
    decoration={
      markings,
      mark=at position 0.7 with {\arrow[scale=1.5]{>}}
        }
    },
    midarrowLong/.style={
    postaction={decorate},
    decoration={
      markings,
      mark=at position 0.57 with {\arrow[scale=1.5]{>}}
        }
    }
    ]
    \node[text width=8cm] at (-0.6, 3.6) {Proof of (\ref{PDG:coherence}) for $(i,j,k)=(1,3,4)$:};

    \node (v1)  at (-3,0)		    [vert,label=below:$b_4$]{};
    \node (v2)  at (0,0)			[vert,label=below:$b_1$]{};

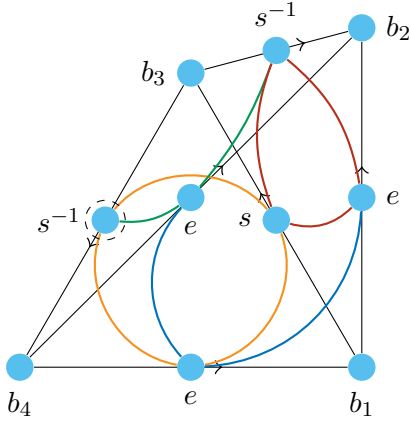
    \node (v3)  at (-1.5,2.6)		[vert,label=left:$b_3$]{};
    \node (v4)  at (0,3)            [vert,label=right:$b_2$]{};
   
    \draw[midarrow]         (v1) -- (v2);
    \draw[midarrow]         (v3) -- (v1);
    \draw[midarrow]         (v2) -- (v3);
    \draw[midarrow]         (v2) -- (v4);
    \draw[midarrow]         (v1) -- (v4);
    \draw[midarrowShort]    (v3) -- (v4);

    \node (m23) at (-0.75,1.3)      [vert,label=left:$s$]{};
    \node (m13) at (-2.25,1.3)      [vert,label=left:$s^{-1}$]{};
    \node (m12) at (-1.5,0)		    [vert,label=below:$e$]{};
    
    \node (m24) at (0,1.5)          [vert,label=right:$e$]{};
    \node (m14) at (-1.5,1.5)       [vert,label=below:$e$]{};
    \node (m34) at (-0.75,2.8)      [vert,label=above:$s^{-1}$]{};
	    
    \draw[BurntOrange,thick]
        (m23) to [bend left=50] (m12) 
        (m12) to [bend left=50] (m13) 
        (m13) to [bend left=50] (m23);
    \draw[RoyalBlue,thick]
        (m12) to [bend left=42] (m14)
        (m24) to [bend left=42] (m12);
    \draw[ForestGreen,thick]
        (m14) to [bend right=10] (m34)
        (m13) to [bend right=20] (m14);
    \draw[BrickRed,thick]
        (m34) to [bend left=20] (m24)
        (m24) to [bend left=30] (m23)
        (m23) to [bend left=20] (m34);
    \draw[dashed] (m13) circle[radius=0.175];
\end{tikzpicture} & \begin{tikzpicture}[scale=1.5,
	vert/.style = {
		circle,
		minimum size = 2mm,
		fill=Cerulean!60,
		font=\itshape
	},
    midarrow/.style={
    postaction={decorate},
    decoration={
      markings,
      mark=at position 0.6 with {\arrow[scale=1.5]{>}}
        }
    },
    midarrowShort/.style={
    postaction={decorate},
    decoration={
      markings,
      mark=at position 0.7 with {\arrow[scale=1.5]{>}}
        }
    },
    midarrowLong/.style={
    postaction={decorate},
    decoration={
      markings,
      mark=at position 0.57 with {\arrow[scale=1.5]{>}}
        }
    }
    ]
    \node[text width=8cm] at (-0.6, 3.6) {Proof of (\ref{PDG:coherence}) for $(i,j,k)=(4,1,3)$:};

    \node (v1)  at (-3,0)		    [vert,label=below:$b_3$]{};
    \node (v2)  at (0,0)			[vert,label=below:$b_4$]{};

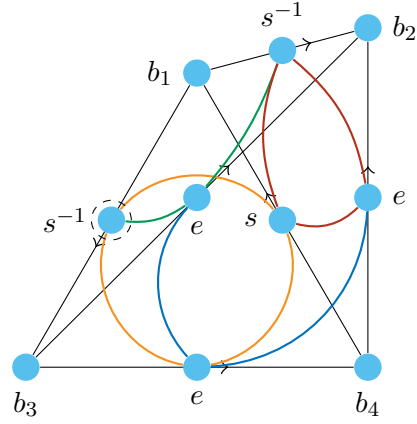
    \node (v3)  at (-1.5,2.6)		[vert,label=left:$b_1$]{};
    \node (v4)  at (0,3)            [vert,label=right:$b_2$]{};

    \draw[midarrow]         (v1) -- (v2);
    \draw[midarrow]         (v3) -- (v1);
    \draw[midarrow]         (v2) -- (v3);
    \draw[midarrow]         (v2) -- (v4);
    \draw[midarrow]         (v1) -- (v4);
    \draw[midarrowShort]    (v3) -- (v4);
    
    \node (m23) at (-0.75,1.3)      [vert,label=left:$s$]{};
    \node (m13) at (-2.25,1.3)      [vert,label=left:$s^{-1}$]{};
    \node (m12) at (-1.5,0)		    [vert,label=below:$e$]{};
    
    \node (m24) at (0,1.5)          [vert,label=right:$e$]{};
    \node (m14) at (-1.5,1.5)       [vert,label=below:$e$]{};
    \node (m34) at (-0.75,2.8)      [vert,label=above:$s^{-1}$]{};
    
    \draw[BurntOrange,thick]
        (m23) to [bend left=50] (m12) 
        (m12) to [bend left=50] (m13) 
        (m13) to [bend left=50] (m23);
    \draw[RoyalBlue,thick]
        (m12) to [bend left=42] (m14)
        (m24) to [bend left=42] (m12);
    \draw[ForestGreen,thick]
        (m14) to [bend right=10] (m34)
        (m13) to [bend right=20] (m14);
    \draw[BrickRed,thick]
        (m34) to [bend left=20] (m24)
        (m24) to [bend left=30] (m23)
        (m23) to [bend left=20] (m34);
    \draw[dashed] (m13) circle[radius=0.175];
\end{tikzpicture}\tabularnewline
\begin{tikzpicture}[scale=1.5,
	vert/.style = {
		circle,
		minimum size = 2mm,
		fill=Cerulean!60,
		font=\itshape
	},
    midarrow/.style={
    postaction={decorate},
    decoration={
      markings,
      mark=at position 0.6 with {\arrow[scale=1.5]{>}}
        }
    },
    midarrowShort/.style={
    postaction={decorate},
    decoration={
      markings,
      mark=at position 0.7 with {\arrow[scale=1.5]{>}}
        }
    },
    midarrowLong/.style={
    postaction={decorate},
    decoration={
      markings,
      mark=at position 0.57 with {\arrow[scale=1.5]{>}}
        }
    }
    ]

    \node[text width=8cm] at (-0.6, 3.6) {Proof of (\ref{PDG:coherence}) for $(i,j,k)=(3,4,2)$:};
    \node (v1)  at (-3,0)		    [vert,label=below:$b_4$]{};
    \node (v2)  at (0,0)			[vert,label=below:$b_2$]{};

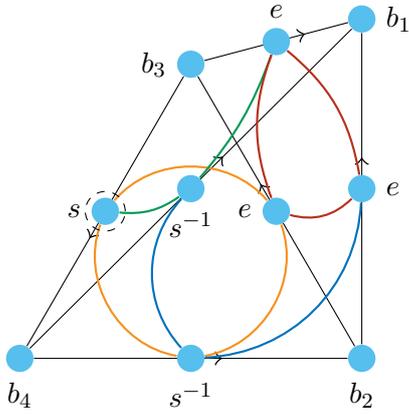
    \node (v3)  at (-1.5,2.6)		[vert,label=left:$b_3$]{};
    \node (v4)  at (0,3)            [vert,label=right:$b_1$]{};

    \draw[midarrow]         (v1) -- (v2);
    \draw[midarrow]         (v3) -- (v1);
    \draw[midarrow]         (v2) -- (v3);
    \draw[midarrow]         (v2) -- (v4);
    \draw[midarrow]         (v1) -- (v4);
    \draw[midarrowShort]    (v3) -- (v4);
    
    \node (m23) at (-0.75,1.3)      [vert,label=left:$e$]{};
    \node (m13) at (-2.25,1.3)      [vert,label=left:$s$]{};
    \node (m12) at (-1.5,0)		    [vert,label=below:$s^{-1}$]{};
    
    \node (m24) at (0,1.5)          [vert,label=right:$e$]{};
    \node (m14) at (-1.5,1.5)       [vert,label=below:$s^{-1}$]{};
    \node (m34) at (-0.75,2.8)      [vert,label=above:$e$]{};
    
    \draw[BurntOrange,thick]
        (m23) to [bend left=50] (m12) 
        (m12) to [bend left=50] (m13) 
        (m13) to [bend left=50] (m23);
    \draw[RoyalBlue,thick]
        (m12) to [bend left=42] (m14)
        (m24) to [bend left=42] (m12);
    \draw[ForestGreen,thick]
        (m14) to [bend right=10] (m34)
        (m13) to [bend right=20] (m14);
    \draw[BrickRed,thick]
        (m34) to [bend left=20] (m24)
        (m24) to [bend left=30] (m23)
        (m23) to [bend left=20] (m34);
    \draw[dashed] (m13) circle[radius=0.175];
\end{tikzpicture} & \begin{tikzpicture}[scale=1.5,
	vert/.style = {
		circle,
		minimum size = 2mm,
		fill=Cerulean!60,
		font=\itshape
	},
    midarrow/.style={
    postaction={decorate},
    decoration={
      markings,
      mark=at position 0.6 with {\arrow[scale=1.5]{>}}
        }
    },
    midarrowShort/.style={
    postaction={decorate},
    decoration={
      markings,
      mark=at position 0.7 with {\arrow[scale=1.5]{>}}
        }
    },
    midarrowLong/.style={
    postaction={decorate},
    decoration={
      markings,
      mark=at position 0.57 with {\arrow[scale=1.5]{>}}
        }
    }
    ]

    \node[text width=8cm] at (-0.6, 3.6) {Proof of (\ref{PDG:coherence}) for $(i,j,k)=(4,2,3)$:};
    \node (v1)  at (-3,0)		    [vert,label=below:$b_4$]{};
    \node (v2)  at (0,0)			[vert,label=below:$b_2$]{};

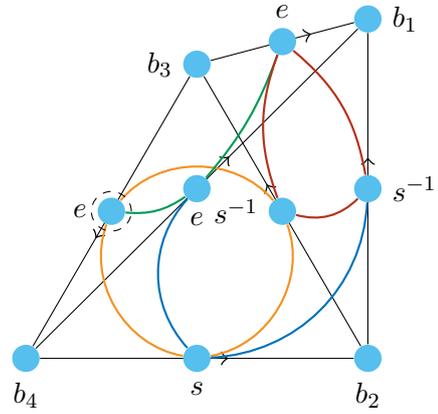
    \node (v3)  at (-1.5,2.6)		[vert,label=left:$b_3$]{};
    \node (v4)  at (0,3)            [vert,label=right:$b_1$]{};

    \draw[midarrow]         (v1) -- (v2);
    \draw[midarrow]         (v3) -- (v1);
    \draw[midarrow]         (v2) -- (v3);
    \draw[midarrow]         (v2) -- (v4);
    \draw[midarrow]         (v1) -- (v4);
    \draw[midarrowShort]    (v3) -- (v4);
    
    \node (m23) at (-0.75,1.3)      [vert,label=left:$s^{-1}$]{};
    \node (m13) at (-2.25,1.3)      [vert,label=left:$e$]{};
    \node (m12) at (-1.5,0)		    [vert,label=below:$s$]{};
    
    \node (m24) at (0,1.5)          [vert,label=right:$s^{-1}$]{};
    \node (m14) at (-1.5,1.5)       [vert,label=below:$e$]{};
    \node (m34) at (-0.75,2.8)      [vert,label=above:$e$]{};
    
    \draw[BurntOrange,thick]
        (m23) to [bend left=50] (m12) 
        (m12) to [bend left=50] (m13) 
        (m13) to [bend left=50] (m23);
    \draw[RoyalBlue,thick]
        (m12) to [bend left=42] (m14)
        (m24) to [bend left=42] (m12);
    \draw[ForestGreen,thick]
        (m14) to [bend right=10] (m34)
        (m13) to [bend right=20] (m14);
    \draw[BrickRed,thick]
        (m34) to [bend left=20] (m24)
        (m24) to [bend left=30] (m23)
        (m23) to [bend left=20] (m34);
    \draw[dashed] (m13) circle[radius=0.175];
\end{tikzpicture}\tabularnewline
\end{tabular}

Clockwise, from the top left, these diagrams prove (\ref{PDG:coherence})
for the index triples $\left(1,3,4\right)$, $\left(4,1,3\right)$,
$\left(4,2,3\right)$, and $\left(3,4,2\right)$. Reading the diagram
in this order, at each step the assumptions for Desargues' theorem
follow from what has already been proven, namely (\ref{PDG:coherence})
for all index triples not involving $\left\{ 3,4\right\} $, as well
as for the index triples $\left(2,3,4\right)$, $\left(4,3,2\right)$,
$\left(3,4,1\right)$, and $\left(1,4,3\right)$. For each of the
four diagrams, we check the assumption of \cref{lem:PDG_Desargues}
involving the indices $\left\{ 3,4\right\} $:
\begin{enumerate}
\item In the top left diagram, one has to verify that $\tilde{f}\left(e_{4,2},s_{3,2}^{-1},s_{3,4}^{-1}\right)=2$,
and rewriting the left hand side as $\tilde{f}\left(s_{2,3},s_{3,4}^{-1},e_{4,2}\right)$,
this follows precisely from (\ref{PDG:coherence}) for $s$ with the
index triple $(2,3,4)$.
\item In the top right diagram, one has to verify that $\tilde{f}\left(e_{2,3},e_{3,4},e_{4,2}\right)=2$.
This follows from (\ref{PDG:coherence}) for $e$ with the index triple
$(2,3,4)$.
\item In the bottom right diagram, one has to verify that $\tilde{f}\left(e_{1,3},e_{4,3},e_{4,1}\right)=2$.
Rewriting the left hand side as $\tilde{f}\left(e_{1,3},e_{3,4},e_{4,1}\right)$,
this follows from (\ref{PDG:coherence}) for $s$ with indices $\left(1,3,4\right)$
(the conclusion from the top left diagram).
\item In the bottom left diagram, one has to verify that $\tilde{f}\left(s_{3,4},s_{4,1}^{-1},e_{1,3}\right)=2$.
This is (\ref{PDG:coherence}) for $s$ with indices $\left(3,4,1\right)$,
which is already known.
\end{enumerate}
\end{proof}

\section{The main undecidability result}

\label[section]{sec:undecidability}

We recall the necessary results from \cite{KY22} and prove \cref{thm:undecidability},
which we restate here for the reader's convenience:
\begin{thm*}
There is no decision procedure that takes as input $n\in\mathbb{N}$
and a vector $v\in\mathbb{Z}^{2^{n}}$ and determines whether $v\in\overline{\Gamma_{n}^{*}}$.
\end{thm*}
The relevant results from \cite{KY22} are on almost multilinear matroids.
These are a subclass of the class of almost entropic polymatroids,
and we briefly introduce them below.

\subsection{Linear polymatroids and almost-multilinear matroids}
\begin{defn}
A polymatroid $\left(E,f\right)$ is linear if there exists a vector
space $V$, vector spaces $\left\{ W_{e}\right\} _{e\in E}$, and
linear maps $\left\{ T_{e}:V\rightarrow W_{e}\right\} _{e\in E}$
satisfying 
\[
f\left(S\right)=\mathrm{rk}\left(T_{S}\right)\text{ for all }S\subseteq E,
\]
where $T_{S}:V\rightarrow\bigoplus_{e\in S}W_{e}$ is the map given
by $T_{S}(v)=\left(T_{e}(v)\right)_{e\in S}$.
\end{defn}

(This is dual to a commonly given definition: here we describe the
linear structure in terms of images of $V$, but often the definition
is given in terms of the dimensions of subspaces of a given space.
The equivalence is proved in \cite[Sec. 9.1]{KY22}.)

It is known that (finite) linear polymatroids are almost entropic.
There is a ``standard proof'' of this that has been rewritten several
times, and it can be found (for instance) in the introduction to \cite{DFZ_5_vars}.
More precisely, that paper proves that if $\left(E,f\right)$ is linear
then $\left(E,\alpha\cdot f\right)$ is entropic for some real $\alpha>0$.
Since the almost entropic cones are convex cones containing the origin,
$\left(E,f\right)$ is almost entropic.
\begin{defn}
A matroid $\left(E,f\right)$ is almost multilinear if for every $\varepsilon>0$
there exists a linear polymatroid $\left(E,f'\right)$ (on the same
set $E$) and a $c\in\mathbb{N}$ such that 
\[
\max_{S\subseteq E}\left|f\left(S\right)-\frac{1}{c}f^{\prime}\left(S\right)\right|<\varepsilon.
\]
\end{defn}

\begin{lem}
A finite almost multilinear matroid $\left(E,f\right)$ is almost
entropic.
\end{lem}

\begin{proof}
Let $\varepsilon>0$. There exists $c>0$ and a linear polymatroid
$\left(E,f'\right)$ such that $\max_{S\subseteq E}\left|f\left(S\right)-\frac{1}{c}f^{\prime}\left(S\right)\right|<\varepsilon$.
Since $\left(E,f'\right)$ is almost entropic, so is $\left(E,\frac{1}{c}f'\right)$,
and we can find an entropic polymatroid $\left(E,g\right)$ such that
$\max_{S\subseteq E}\left|g\left(S\right)-\frac{1}{c}f'\left(S\right)\right|<\varepsilon$.
In particular, 
\[
\max_{S\subseteq E}\left|g\left(S\right)-f\left(S\right)\right|<2\varepsilon,
\]
so $f$ is a pointwise limit of entropic polymatroids.
\end{proof}

\subsection{Results from \cite{KY22}}

\label[section]{sec:KY22}

In outline, the proof in \cite{KY22} that there is no decision procedure
to determine whether a matroid is almost multilinear proceeds as follows:
\begin{enumerate}
\item From a torsion-free finitely presented sofic group $G=\left\langle S\mid R\right\rangle $
(with the given presentation symmetric triangular, see \cref{cons:fp_group_PDG})
and an element $s\in S$ we construct a certain family $\mathcal{F}$
of PDGs of rank $3$. All of them are PDGs of finitely presented groups
which are quotients $Q$ of $G$, on the same set of generators as
$S$ (modulo an equivalence relation, identifying pairs of generators
that map to the same element of $Q$. The generator $s$ is never
identified with the generator $e\in S$ mapping to the trivial element
$e_{Q}$).
\item It is proved that $s$ does not map to the trivial element of $G$
(and hence $s$ does not map to the trivial element of $G$) if and
only if at least one matroid in $\mathcal{F}$ is almost multilinear.
\end{enumerate}
The family $\mathcal{F}$ in step (1) is always finite and computable
given $\left\langle S\mid R\right\rangle $ and $s\in S$. The two
implications of part (2) are separate theorems in \cite{KY22}:
\begin{enumerate}
\item [(a)]The statement that if $s$ maps to a nontrivial element in $G$
then at least one matroid in $\mathcal{F}$ is almost multilinear
is \cite[Thm. 9.16]{KY22}. Note that in this case also at least one
matroid in $\mathcal{F}$ is almost entropic.
\end{enumerate}
\begin{itemize}
\item [(b)]The statement that if at least one matroid in $\mathcal{F}$
is almost multilinear then $s$ does not map to the trivial element
of $G$ is \cite[Thm. 9.17]{KY22}. (In \cite{KY22} we did not have
the machinery to prove this for almost entropic matroids.)
\end{itemize}
Part (a) is a fortiori true for almost entropic matroids, because
an almost multilinear matroid is almost entropic. Hence, to prove
step (2) for almost entropic (rather than almost multilinear) matroids
it suffices to prove part (b) with ``almost entropic'' replacing
``almost multilinear'' in the statement. 
\begin{rem}
This is really just an extension of \cite[Thm. 9.12]{KY22} to the
almost entropic case - the rest of the argument is identical. During
the writing of \cite{KY22} we spent some time thinking about this,
but did not have the correct tools to tackle the problem. Our argument
extracted a concrete group representation of $G$ (taking values in
``approximate linear maps'') from the data of an almost multilinear
approximating sequence. Getting such explicit representations in the
almost entropic setting seems challenging, but Desargues' theorem
gives a viable replacement.
\end{rem}

\begin{proof}
[Proof of \cref{thm:undecidability}] Following the discussion above,
(and re-setting the notation, so $G=\left\langle S\mid R\right\rangle $
denotes a general finitely presented group,) it suffices to prove:
if $\left(E,f,S\right)$ is the PDG of rank $3$ of a finitely presented
group $G=\left\langle S\mid R\right\rangle $ (see \cref{cons:fp_group_PDG}),
and $x\in S$ satisfies $f\left(x_{1,2},e_{1,2}\right)\neq1$, then
$x$ maps to a nontrivial element of $G$. 

By \cref{thm:lifting}, the PDG of rank $4$ of the same finitely
presented group is also almost entropic. Call this $\left(\hat{E},\hat{f},S\right)$.
Now by \cref{thm:group_from_PDG} we have an extension $\left(\tilde{E},\tilde{f},\tilde{S}\right)$
such that $S\subseteq\tilde{S}$, and the parallelism classes of $\tilde{S}$
form a group. By definition, the geometric product in $\left(\tilde{E},\tilde{f},\tilde{S}\right)$
agrees with that in $\left(\hat{E},\hat{f},S\right)$ whenever the
latter is defined. In particular, for each relator $s\cdot s'\cdot s''$
in $R$ the same elements of $\tilde{S}$ satisfy $\left[s\right]\cdot\left[s'\right]\cdot\left[s''\right]=\left[e\right]$.
It follows that the group $\left(\tilde{S}/\mathord{\sim},\cdot\right)$
is a quotient of $G=\left\langle S\mid R\right\rangle $, where the
quotient map sends each $t\in S$ to $\left[t\right]$. Since $f\left(x_{1,2},e_{1,2}\right)\neq1$,
also $\tilde{f}\left(x_{1,2},e_{1,2}\right)\neq1$ and hence $\left[x\right]\neq\left[e\right]$.
Hence $x$ maps to a nontrivial element of $G$.
\end{proof}
\bibliographystyle{amsalpha}
\bibliography{aent}

\end{document}